\documentclass[12pt,a4paper,reqno,final]{amsart} 
\usepackage{amsmath, amssymb, amsfonts, amsthm, verbatim, dsfont, graphicx, enumerate, pifont,tikz} 
\usepackage{hyperref}
\usepackage[notref,notcite]{showkeys} 
\usepackage[all,line,dvips]{xy}
\usepackage{amsaddr}
\theoremstyle{theorem}
\newtheorem{theorem}{Theorem}
\newtheorem{proposition}[theorem]{Proposition}
\newtheorem{lemma}[theorem]{Lemma}
\newtheorem{corollary}[theorem]{Corollary}
\theoremstyle{definition}
\newtheorem{definition}[theorem]{Definition}
\newtheorem{remark}[theorem]{Remark}

\newtheorem{example}[theorem]{Example}

\theoremstyle{plain}

\newcommand{\Z}{\mathbb Z} 
\newcommand{\R}{\mathbb R} 
\newcommand{\C}{\mathbb C}

\DeclareMathOperator{\Graph}{Graph}

\DeclareMathOperator{\supp}{supp}

\DeclareMathOperator{\PSS}{PSS}

\DeclareMathOperator{\id}{id}
\DeclareMathOperator{\crit}{Crit}

\DeclareMathOperator{\image}{Image}
\DeclareMathOperator{\Ham}{Ham}
\DeclareMathOperator{\Symp}{Symp}

\title{Invariants of Lagrangian cobordisms via spectral numbers}
\author{Mads R. Bisgaard}
\address{Department of Mathematics, ETH Z{\"u}rich, R{\"a}mistrasse 101\\
8092 Z{\"u}rich,
Switzerland\\
\emph{\href{mailto:mads.bisgaard@math.ethz.ch}{mads.bisgaard@math.ethz.ch}}}
\begin{document}
\maketitle
\begin{abstract}
We extend parts of the Lagrangian spectral invariants package recently developed by Leclercq and Zapolsky to the theory of Lagrangian cobordism developed by Biran and Cornea. This yields a non-degenerate Lagrangian "spectral metric" which bounds the Lagrangian "cobordism metric" (recently introduced by Cornea and Shelukhin) from below. It also yields a new numerical Lagrangian cobordism invariant as well as new ways of computing certain asymptotic Lagrangian spectral invariants explicitly.
\end{abstract}

\section{Introduction}
One approach to studying Lagrangian submanifolds of a symplectic manifold $(M,\omega)$ which has attracted a lot of attention lately is by studying their Lagrangian cobordisms in $(\R^2 \times M,\omega_{\R^2}\oplus \omega)$ (see precise definitions in Section \ref{secLagCob}). Biran and Cornea showed in \cite{BiranCornea13} and \cite{BiranCornea14} that suitable Lagrangian cobordisms preserve symplectic invariants. Considering "Lagrangians up to Lagrangian cobordism" thus seems like a very natural notion for studying the \emph{symplectic} topology of Lagrangian submanifolds. Moreover, Cornea and Shelukhin \cite{CorneaShelukhin15} recently discovered the existence of a remarkable "cobordism metric" $d_c$ on suitable subspaces of the space of Lagrangians in $(M,\omega)$ (see precise definitions in Section \ref{secLagmet}). This metric can be viewed as a generalization of the Lagrangian version of the well-known Hofer metric introduced by Chekanov \cite{Chekanov00}. Motivated by these discoveries we produce a Lagrangian cobordism invariant by applying the Lagrangian spectral invariant package recently developed by Leclercq and Zapolsky in \cite{LeclercqZapolsky15}. We also show that spectral numbers provide functions on subsets of the space of Lagrangian submanifolds in $(M,\omega)$ to $\R$ which are Lipschitz continuous with respect to $d_c$ and use this to define a non-degenerate "spectral metric" which bounds $d_c$ from below.


\subsection{Setting and notation}
\label{setting}
Throughout the paper we consider a connected symplectic manifold $(M^{2n},\omega)$ which is either closed or open and convex at infinity \cite{GromovEliashberg91}. We also consider the associated symplectic manifold $(\tilde{M},\tilde{\omega})$ defined by $\tilde{M}:=\R^2(x,y) \times M$ and $\tilde{\omega}:=\omega_{\R^2}\oplus \omega$, where $\omega_{\R^2}:=dx\wedge dy$. Unless otherwise stated, any Lagrangian submanifold $L^n\subset (M,\omega)$ will be assumed closed, connected and \emph{monotone}. By this we mean that there is a positive constant $\tau_L>0$ satisfying 
\[
\omega|_{\pi_2(M,L)}=\tau_L \cdot  \mu|_{\pi_2(M,L)},
\] 
where $\omega:\pi_2(M,L) \to \R$ denotes integration of $\omega$ and $\mu: \pi_2(M,L)\to \Z$ denotes the Maslov index. We will also assume that the minimal Maslov number $N_L:=\min \{\mu(\alpha)>0\ |\ \alpha \in \pi_2(M,L) \}$ associated to $L$ satisfies $N_L \geq 2$. If $\mu|_{\pi_2(M,L)}\equiv 0$ we set $N_L=\infty$ (this is a special case of the \emph{weakly exact} setting - see Section \ref{adapcob}). 

In the monotone setting both Floer homology $HF_*(L)$ and quantum homology $QH_*(L)$ of $L$ with $\Lambda$-coefficients are well-defined, where $\Lambda:=\Z_2[t,t^{-1}]$ \cite{BiranCornea07}, \cite{BiranCornea09}, \cite{LeclercqZapolsky15}, \cite{Zapolsky15}.\footnote{\cite{BiranCornea07} and \cite{BiranCornea09} use language/notation which is slightly different from \cite{LeclercqZapolsky15} and \cite{Zapolsky15}. To avoid confusion we therefore point out that $QH(L)$ (respectively $HF(L)$) with $\Lambda$-coefficients in \cite{BiranCornea07} and \cite{BiranCornea09} corresponds to $QH(L)$ (respectively $HF(L)$) of a suitable quotient complex with $\Z_2$-coefficients in \cite{LeclercqZapolsky15} and \cite{Zapolsky15} (see Section 2.6 in \cite{LeclercqZapolsky15} for details).} Here we will only work with $\Lambda$-coefficients and therefore omit them from the notation.

Given $\tau>0$ we denote by $\mathcal{L}_{\tau}=\mathcal{L}_{\tau}(M,\omega)$ the space of all Lagrangian submanifolds $L\subset (M,\omega)$ as above, satisfying the additional condition that $\tau_L=\tau$ (i.e. all Lagrangians are \emph{uniformly monotone}). We also denote by $\mathcal{L}^*_{\tau}\subset \mathcal{L}_{\tau}$ the subspace consisting of Lagrangians $L$ for which $QH_*(L)\neq 0$.

\section{Main results}
\label{MainResults}

\subsection{The Lagrangian cobordism metric structure}
\label{secLagmet}
Given a Lagrangian cobordism $V\subset (\tilde{M},\tilde{\omega})$ connecting two Lagrangians $L,L'\in \mathcal{L}_{\tau}$ (see Section \ref{secLagCob} for precise definitions) it is natural to consider the set $\pi(V)\subset \R^2$, where $\pi:\tilde{M}=\R^2 \times M \to \R^2$ denotes the projection. The insight that $\pi(V)$ contains valuable information originally arose during Biran and Cornea's extensive study of Lagrangian cobordism \cite{BiranCornea13}, \cite{BiranCornea14}. The idea was made quantitative in \cite{CorneaShelukhin15} where Cornea and Shelukhin established the existence of a remarkable natural \emph{cobordism metric} on the space $\mathcal{L}_{\tau}$.\footnote{In fact Cornea and Shelukhin showed that there is a cobordism metric in several different settings.} We say that the Lagrangian cobordism $V$ is \emph{elementary} if it is connected and monotone with $N_V\geq 2$ when viewed as a Lagrangian submanifold of $(\tilde{M},\tilde{\omega})$ (see also Section \ref{secLagCob}).
\begin{definition}[\cite{CorneaShelukhin15}]
	Given $L,L'\subset \mathcal{L}_{\tau}$ as well as an elementary Lagrangian cobordism $V:L' \rightsquigarrow L$ the \emph{outline of $V$}, $ou(V)$, is by definition the closed subset of $\R^2$ given as the complement of the union of unbounded components of $\R^2 \backslash \pi(V)$. The \emph{shadow of $V$} is defined by
	\[
	\mathcal{S}(V):=\text{Area}(ou(V)).
	\]
\end{definition}
The main result of this paper shows that $\mathcal{S}(V)$ provides a natural upper bound on the difference in spectral numbers coming from the bounding Lagrangians. The theory of spectral numbers for Lagrangian submanifolds has been developed in various settings by various authors, starting with Viterbo \cite{Viterbo92} and Oh \cite{Oh97}, \cite{Oh99}. Here we will use the version for monotone Lagrangians recently developed by Leclercq and Zapolsky \cite{LeclercqZapolsky15}. Associated to a given $L\in \mathcal{L}_{\tau}$ they defined a \emph{spectral invariant} function
\[
l_L:QH_*(L)\times \widetilde{\Ham}(M,\omega)\to \R \cup \{-\infty\}
\] 
satisfying $l(\alpha,\phi)=-\infty$ if and only if $\alpha =0\in QH_*(L)$ (see Section \ref{SecLagInv} for preliminaries on Lagrangian spectral invariants).
To state our main result we recall that, if $L,L'\in \mathcal{L}_{\tau}$ and $V:L' \rightsquigarrow L$ is an elementary Lagrangian cobordism then, by Theorem 2.2.2 in \cite{BiranCornea13}, $V$ induces a \emph{ring isomorphism} $\Phi_V:QH_*(L)\stackrel{\cong}{\to} QH_*(L')$. 
\begin{theorem}
	\label{lemcob1}
	Let $L,L'\in \mathcal{L}_{\tau}$ and let $V:L' \rightsquigarrow L$ be an elementary Lagrangian cobordism. Then
	\[
	|l_L(\alpha ,\phi)-l_{L'}(\Phi_V(\alpha),\phi)|\leq \mathcal{S}(V)
	\]
	for all $\alpha \in QH_*(L)\backslash \{0\}$ and all $\phi \in \widetilde{\Ham}(M,\omega)$.
\end{theorem} 
Example \ref{cobex1} below shows that the result is sharp in the sense that there exist cobordisms $V$ for which the statement becomes false if $\mathcal{S}(V)$ is replaced by a smaller number.
Cornea and Shelukhin further considered the following 
\begin{definition}[\cite{CorneaShelukhin15}]
	Define a function $d_c:\mathcal{L}_{\tau}\times \mathcal{L}_{\tau}\to [0,\infty]$ by
	\[
	d_c(L,L'):=\inf\{\mathcal{S}(V)\ |\ V:L' \rightsquigarrow L \}, \quad L,L'\in \mathcal{L}_{\tau}.
	\]
	Here the infimum runs over all elementary Lagrangian cobordisms $V:L' \rightsquigarrow L$.
\end{definition}
One of the main results in Cornea and Shelukhin's paper \cite{CorneaShelukhin15} is that $d_c$ in fact is a metric on $\mathcal{L}_{\tau}$. Of course $d_c(L,L')=\infty$ if and only if there do not exist any elementary Lagrangian cobordisms $V:L' \rightsquigarrow L$. Recall that, given $L\in \mathcal{L}_{\tau}$, the fundamental class $[L]\in QH_n(L)$ is the unity with respect to the ring structure on $QH_*(L)$. It is convenient to introduce the notation $l^+_L(\phi):=l_L([L],\phi)$ for $\phi \in \widetilde{\Ham}(M,\omega)$. Since $\Phi_V$ preserves the ring structure, Theorem \ref{lemcob1} implies 
\begin{corollary}
	\label{corcob1}
	For every pair of Lagrangian submanifolds $L,L'\in \mathcal{L}^*_{\tau}$ we have
	\[
	|l_L^{+}(\phi)-l_{L'}^{+}(\phi)|\leq d_c(L,L') \quad \forall \ \phi \in \widetilde{\Ham}(M,\omega).
	\]
\end{corollary}
\begin{remark}
Corollary \ref{corcob1} says that, for any fixed $\phi \in \widetilde{\Ham}(M,\omega)$, the function $(\mathcal{L}^*_{\tau},d_c)\to (\R ,|\cdot| )$ given by $L'\mapsto l^+_{L'}(\phi)$ is $1$-Lipschitz. Of course this statement is only interesting when $L'\mapsto l^+_{L'}(\phi)$ is restricted to an elementary Lagrangian cobordism class $\subset \mathcal{L}^*_{\tau}$.  
\end{remark}

The following definition and proposition were generously suggested to us by the anonymous referee whom we wholeheartedly thank!
\begin{definition}
\label{defref}
Define $d_s:\mathcal{L}^*_{\tau}\times \mathcal{L}^*_{\tau}\to [0,\infty]$ by 
\[
d_s(L,L'):= \sup \{ |l^+_L(\phi)-l^+_{L'}(\phi)|\ |\ \phi \in \widetilde{\Ham}(M,\omega)\}, \quad L,L'\in \mathcal{L}_{\tau}^*.
\]
\end{definition}
As the notation suggests $d_s$ is a (spectral) metric. It is clear that $d_s$ is symmetric and satisfies the triangle inequality. Hence, the only non-trivial property to check in order for $d_s$ to be a metric, is non-degeneracy.
\begin{proposition}
\label{propref}
$d_s:\mathcal{L}^*_{\tau}\times \mathcal{L}^*_{\tau}\to [0,\infty]$ is a non-degenerate metric.
\end{proposition} 
Note that this result shows that the estimate in Theorem \ref{lemcob1} is non-trivial whenever $L\neq L'$.
As an immediate consequence of Corollary \ref{corcob1} we obtain
\begin{corollary}
\label{corref}
For all $L,L'\in \mathcal{L}^*_{\tau}$ we have $d_s(L,L')\leq d_c(L,L')$.
\end{corollary}
This result together with Proposition \ref{propref} gives a new proof of the fact that the restriction of $d_c$ to $\mathcal{L}^{*}_{\tau}$ is non-degenerate. Fixing $L\in \mathcal{L}^*_{\tau}$ one often considers the subset $\mathcal{H}(L):=\{\phi(L)\ |\ \phi \in \Ham(M,\omega)\}\subset \mathcal{L}^*_{\tau}$ equipped with the Hofer metric $d_H$ \cite{Chekanov00}. It follows from Corollary \ref{corref} and the Lagrangian suspension construction \cite{CorneaShelukhin15} that
\begin{equation}
\label{cobref1}
d_s(L',L'')\leq d_c(L',L'')\leq d_H(L',L'') \quad \forall \ L',L''\in \mathcal{H}(L).
\end{equation}
Denote by $\Symp_{c|M}(\tilde{M}, \tilde{\omega} )\leq \Symp(\tilde{M},\tilde{\omega} )$ the subgroup of symplectomorphisms $\psi$ which are \emph{compactly supported relative to $M$} in the sense that there is a compact $K\subset \R^2$ such that $\psi|_{(\R^2 \backslash K)\times M}=(\id_{\R^2}\times \varphi)|_{(\R^2 \backslash K)\times M}$ for some $\varphi \in \Symp(M,\omega)$. Given $L,L'\in \mathcal{L}^*_{\tau}$, $\psi \in \Symp_{c|M}(\tilde{M},\tilde{\omega})$ and an elementary Lagrangian cobordism $V:L' \rightsquigarrow L$ we obtain a new cobordism $\psi(V):\varphi(L') \rightsquigarrow \varphi(L)$. It is easy to check that $d_s$ is $\Symp(M,\omega)$-invariant and therefore Corollary \ref{corref} gives a Lagrangian non-squeezing type inequality:
\begin{equation}
\label{cobref2}
d_s(L,L') \leq \mathcal{S}(\psi(V)) \quad \forall \ \psi \in \Symp_{c|M}(\tilde{M}, \tilde{\omega}).
\end{equation}
From this point of view it would be interesting to understand for which pairs $L',L''\in \mathcal{H}(L)$ it holds that $d_s(L',L'')=d_H(L',L'')$. For such pairs one can find Lagrangian suspensions which are (close to) "optimal" in the sense that they (almost) minimize shadow among all elementary Lagrangian cobordisms $L'' \rightsquigarrow L'$ (see also Example \ref{cobex1} below). Some investigations in this direction were already carried out in Remark 5.1 in \cite{CorneaShelukhin15}. Moreover, Corollary \ref{corcob1} above can be viewed as a generalization of a bound found in \cite{LeclercqZapolsky15} for the Hofer distance on the universal cover of $\mathcal{H}(L)$ for a fixed $L\in \mathcal{L}^*_{\tau}$. 

It is an open problem to understand the extend to which $d_c$ differs from $d_H$ in our setting: The main examples of elementary Lagrangian cobordisms are Lagrangian suspensions and their images under $\Symp_{c|M}(\tilde{M},\tilde{\omega})$-elements. On the other hand there are many explicit examples of non-monotone Lagrangian cobordisms which do not arise as Lagrangian suspensions \cite{BiranCornea13}, \cite{Chekanov97}, \cite{Haug15}. In \cite{Bisgaard17} we study how properties similar to (\ref{cobref2}) of such cobordisms are intimately linked to their topology.  
\begin{example}\footnote{We thank the anonymous referee for suggesting this example.}
\label{cobex1}
Consider $T^*S^1 = \R /\Z \times \R$ with coordinates $(q,p)\in \R/\Z\times \R$ and equipped with the symplectic structure $dp \wedge dq$. Denote by $L\subset T^*S^1$ the 0-section. Define an autonomous Hamiltonian $H\in C^{\infty}(T^*S^1)$ by $H(q,p)=(\sin(2\pi q)+1)$ and let $L':=\phi_H^1(L)\subset T^*S^1$. Choose a monotone function $\rho \in C^{\infty}(\R;[0,1])$ such that for some small $\epsilon >0$ we have $\rho =0$ on $(-\infty, \epsilon]$ and $\rho=1$ on $[1-\epsilon,\infty)$. Then the Lagrangian suspension construction \cite{Polterovich01} applied to the time-dependent Hamiltonian $\hat{H}_t(q,p):=\rho'(t)H(q,p)$ produces an exact Lagrangian cobordism $V:L' \rightsquigarrow L$ with shadow $\mathcal{S}(V)=2$. Moreover, since $H$ has Hofer norm $=2$ we conclude that $d_{H}(L,L')\leq 2$. We claim that $d_s(L,L')\geq 2$, which by (\ref{cobref1}) implies 
\[
d_s(L,L')= d_c(L,L')= d_H(L,L')=2.
\]
To see this fix a small $\epsilon >0$ and a corresponding $1> \! \! \! >\delta >0$ such that 
\[
\max\{ H(q,p) \ | \ (q,p)\in T^*S^1\backslash B_{3\delta}(\tfrac{1}{4})\}\geq 2-\frac{\epsilon}{3},
\] 
where $B_{3\delta}(\tfrac{1}{4})$ denotes the ball of radius $3\delta$ centered at $\tfrac{1}{4}\in L$. Choose  $\varphi_1 \in C^{\infty}(L;[0,1])$ such that 
\begin{align*}
	\varphi_1 \left\{
	\begin{array}{ll}
	=0, & \text{on}\ L\cap B_{2\delta}(\tfrac{1}{4}) \\ 
	=1, & \text{on}\ L\backslash B_{3\delta}(\tfrac{1}{4})
	\end{array}
	\right. \quad \& \quad
	\varphi'_1 \left\{
	\begin{array}{ll}
	\leq 0, & \text{on}\ [\frac{1}{4}-3\delta, \frac{1}{4}-2\delta ]  \\ 
	\geq 0, & \text{on}\ [\frac{1}{4}+2\delta , \frac{1}{4}+3\delta],
	\end{array}
	\right.
\end{align*} 
and define $H^1(q,p):=\varphi_1(q)H(q,p)$. Applying Lemma \ref{lemref1} below together with an easy approximation argument one sees that for all $s\in [0,1]$ we have $l^+_L(sH^1)=\max_{T^*S^1}(sH^1) \geq s(2-\tfrac{\epsilon}{3})$. Fix now $\varphi_2 \in C^{\infty}(T^*S^1;[0,1])$ such that $\varphi_2=1$ outside a very small neighborhood of $L'\backslash (B_{\delta}(\tfrac{1}{4})\cup B_{\delta}(\tfrac{3}{4}))$ and $\varphi_2=0$ on an even smaller neighborhood of $L'\backslash (B_{\delta}(\tfrac{1}{4}) \cup B_{\delta}(\tfrac{3}{4}))$. Define $H^2(q,p):=\varphi_2(q,p)H^1(q,p)$. By continuity and the Lagrangian control property from \cite{LeclercqZapolsky15} we have $|l^+_{L'}(sH^2)|\leq \tfrac{\epsilon}{3}$ for all $s\in [0,1]$ if $\delta$ is chosen small enough. For a $s_*\in (0,1)$ very close to $1$ (depending only on the set $\{\varphi_2 \neq 1\}$) the path $\{\phi_{s_*H^1}^t(L)=\phi_{s_*H^2}^t(L)\}_{t \in [0,1]}$ is contained in the set $\{ \varphi_2=1 \}$. In particular, for every $s\in [0,s_*]$, the Hamiltonian chords of the autonomous Hamiltonian $sH^2+(s_*-s)H^1=(s_*+s(\varphi_2-1))H^1$ connecting $L$ to itself coincide with those of the Hamiltonian $s_*H^1$, and are contained in the set $\{\varphi_2=1\}$. Hence, spectrality \cite{LeclercqZapolsky15} gives
\[
l^+_L(sH^2+(s_*-s)H^1)\in \mathcal{A}_{s_*H^1:L}(\crit(\mathcal{A}_{s_*H^1:L})) \quad \forall \ s\in [0,s_*].
\]
Since $\mathcal{A}_{s_*H^1:L}(\crit(\mathcal{A}_{s_*H^1:L}))\subset \R$ is nowhere dense and 
\[
[0,s_*]\ni s \mapsto l^+_L(sH^2+(s_*-s)H^1) \in \R
\] 
is continuous we conclude that $l^+_L(s_*H^2)=l^+_L(s_*H^1)\geq s_*(2-\tfrac{\epsilon}{3})\geq 2-\tfrac{2\epsilon}{3}$ if $s_*$ is sufficiently close to $1$. Hence,
\[
d_s(L,L')\geq |l^+_L(s_*H^2)-l^+_{L'}(s_*H^2)|\geq |l^+_L(s_*H^2)|-|l^+_{L'}(s_*H^2)| \geq 2-\epsilon.
\]
Now the claim follows by letting $\epsilon \to 0$.
\end{example}


\subsection{A Lagrangian cobordism invariant}
\label{Lagspecinv}
Fix $L\in \mathcal{L}^*_{\tau}$. A nice fact about $l_L^+$ is that it satisfies the triangle inequality
\[
l_{L}^+(\phi \psi)\leq l_{L}^+(\phi )+l_{L}^+(\psi)\quad \forall \ \phi, \psi \in \widetilde{\Ham}(M,\omega).
\]
This and the continuity property of $l_L^+$ allows one to consider the asymptotic spectral invariant $\sigma_L:\widetilde{\Ham}(M,\omega) \to \R$ given by
\[
\sigma_L(\phi):=\lim_{k\to \infty} \frac{l_L^+(\phi^k)}{k} , \quad \phi \in \widetilde{\Ham}(M,\omega).
\]
We want to point out that, just like $l_L$, $\sigma_L$ is known to satisfy a number of "nice" properties. In the case when $L$ is the 0-section of a cotangent bundle many of these are documented in \cite{ZapolskyMonznerVichery12}. For the monotone setting we are considering many analogous properties follow immediately from the properties of $l_L$ which are documented in \cite{LeclercqZapolsky15}.

Our next result shows that $\sigma_L$ can be considered as an object associated to $L$'s elementary Lagrangian cobordism class.
\begin{theorem}
\label{thmcob1}
Let $L\in \mathcal{L}^*_{\tau}$. Then the function $\sigma_L:\widetilde{\Ham}(M,\omega)\to \R$ is an elementary Lagrangian cobordism invariant of $L$. In other words, if $L'\in \mathcal{L}_{\tau}$ is in the same elementary Lagrangian cobordism class as $L$ then $\sigma_{L'}:\widetilde{\Ham}(M,\omega)\to \R$ is well-defined and $\sigma_L=\sigma_{L'}$.
\end{theorem}
As mentioned in the introduction Biran and Cornea's Lagrangian cobordism theory \cite{BiranCornea13}, \cite{BiranCornea14} shows that it is desirable to be able to detect whether or not two given Lagrangians are in the same elementary Lagrangian cobordism class.
To our knowledge, $L\mapsto \sigma_L$ is one of very few numerical invariants known for Lagrangian cobordism. Naturally one would like to make use of the algebraic structures on $\widetilde{\Ham}(M,\omega)$ and properties of $\sigma_L$ to derive criteria for detecting the non-existence of elementary Lagrangian cobordisms. One example of how this can be done is the following result. Recall that $\Ham(M,\omega)$ is a normal subgroup of $\Symp(M,\omega)$. In particular $\Symp(M,\omega)$ acts on $\widetilde{\Ham}(M,\omega)$ by conjugation
\begin{equation*}
\begin{array}{c}
\Symp(M,\omega)\times \widetilde{\Ham}(M,\omega) \to \widetilde{\Ham}(M,\omega) \\
(\psi , \phi)  \mapsto \psi \phi \psi^{-1}
\end{array}
\end{equation*}
As a consequence of Theorem \ref{thmcob1} and the symplectic invariance property from \cite{LeclercqZapolsky15} we obtain
\begin{corollary}
\label{cobcor100}
Let $L\in \mathcal{L}^*_{\tau}$ and $\psi \in \Symp(M,\omega)$. If $L$ and $\psi(L)$ are in the same elementary Lagrangian cobordism class then $\sigma_L$ is invariant under conjugation by $\psi$. I.e.
\[
\sigma_L(\phi)=\sigma_L(\psi \phi \psi^{-1}) \quad \forall \ \phi \in \widetilde{\Ham}(M,\omega).
\] 
\end{corollary}
By the Lagrangian suspension construction this result implies in particular that $\sigma_L$ is invariant under conjugation by elements of $\Ham(M,\omega)$. We do not know of any examples $\psi \in \Symp(M,\omega)\backslash \Ham(M,\omega)$ such that $L$ and $\psi(L)$ are in the same elementary Lagrangian cobordism class.
\begin{remark}
Consider a pair of Lagrangians $L,L'\in \mathcal{L}^*_{\tau}$ satisfying $L\cap L'=\emptyset$. Now choose a normalized\footnote{If $(M^{2n},\omega)$ is a closed symplectic manifold we say that a Hamiltonian $H\in C^{\infty}([0,1]\times M)$ is \emph{normalized} if $\int_M H_t \omega^n =0$ for all $t\in [0,1]$. If $(M^{2n},\omega)$ is non-compact we say that $H\in C^{\infty}([0,1]\times M)$ is \emph{normalized} if it has compact support.} autonomous Hamiltonian $H\in C^{\infty}(M)$ satisfying $H|_L \equiv c$ and $H|_{L'} \equiv c'$ for constants $c\neq c'$. Then
\begin{equation*}
	\sigma_L(\phi_H)=c\neq c'=\sigma_{L'}(\phi_H),
\end{equation*} 
by the Lagrangian control property from \cite{LeclercqZapolsky15}. In view of Theorem \ref{thmcob1} this observation gives a new proof of the following result which can also be derived from Biran and Cornea's work \cite{BiranCornea13} (see Remark \ref{cobrem1} below).
\end{remark}
\begin{corollary}
	\label{cobcor191919}
	Let $L\in \mathcal{L}^*_{\tau}$. If $L' \in \mathcal{L}_{\tau}$ is in the same elementary Lagrangian cobordism class as $L$ then $L \cap L' \neq \emptyset$. 
\end{corollary}
In fact there is even a third proof of this fact based on the metric $d_s$: If $L,L'\in \mathcal{L}^*_{\tau}$ satisfy $L\cap L' =\emptyset$ then it is easy to check that $d_s(L,L')=\infty$. In particular Corollary \ref{corref} implies that $d_c(L,L')=\infty$ and therefore $L$ and $L'$ cannot be in the same elementary Lagrangian cobordism class.  
\begin{example}
Spectral invariants coming from Floer theory are known to be very hard to compute. However, given $L\in \mathcal{L}^*_{\tau}$, a consequence of Theorem \ref{thmcob1} is that there is a rather large subset of $\widetilde{\Ham}(M,\omega)$ on which $\sigma_L$ can be computed explicitly! Consider the subgroup $\mathcal{G}_L\subset \widetilde{\Ham}(M,\omega)$ defined by 
\[
\phi \in \mathcal{G}_L \stackrel{Def.}{\Longleftrightarrow} (\exists \{\phi_t\}_{t\in [0,1]} \in \phi \ : \ \phi_t(L)=L\ \forall \ t\in [0,1]).
\]
In other words $\mathcal{G}_L$ consists exactly of the homotopy classes (rel. endpoints) of paths $\phi$ in $\Ham(M,\omega)$, based at the identity, which contain a path $\{\phi_t\}_{t\in [0,1]}$ satisfying $\phi_t(L)=L$ for all $t\in [0,1]$. It is easy to check that $\phi \in \mathcal{G}_L$ if and only if $\phi=\phi_H$ for a normalized $H\in C^{\infty}([0,1]\times M)$ satisfying $H_t|_L=c(t)$ for some $c\in C^{\infty}([0,1])$. For such a Hamiltonian the Lagrangian control property from \cite{LeclercqZapolsky15} reads
\begin{equation}
\label{Cob999}
\sigma_L(\phi_H)=\int_0^1 c(t)dt.
\end{equation}
Applying Theorem \ref{thmcob1} this has the following interesting consequence: $\sigma_L(\phi)$ \emph{can be computed explicitly by a formula similar to (\ref{Cob999}) for every}
\begin{equation}
\label{cobeq99999}
\phi \in \bigcup_{L'}\mathcal{G}_{L'}.
\end{equation}
\emph{Here $L'$ runs over the entire elementary Lagrangian cobordism class of $L$.} Note that, by the Lagrangian suspension construction, the orbit of $L$ under the natural action $\Ham(M,\omega)\times \mathcal{L}_{\tau} \to \mathcal{L}_{\tau}$ is contained in the elementary Lagrangian cobordism class of $L$. It is therefore clear that the union in (\ref{cobeq99999}) is a rather large set in general. 
\end{example}

Analogues of $\sigma_L$ have a very prominent history in symplectic topology. In the case of spectral invariants coming from Hamiltonian Floer homology the study of the analogue of this quantity was pioneered by Entov and Polterovich in their development of \emph{Calabi quasimorphisms} on $\widetilde{\Ham}(M,\omega)$ \cite{EntovPolterovich03} (see also \cite{PolterovichRosen14}). In case $L$ is the zero-section of the cotangent bundle $T^*N$ of a closed manifold $N$ Monzner, Vichery and Zapolsky \cite{ZapolskyMonznerVichery12} showed, using ideas due to Viterbo \cite{Viterbo08}, that $\sigma_L$ is closely related to Mather's $\alpha$-function. Moreover, for the special case $N=\mathbb{T}^n$, they showed that $\sigma_L$ is closely related to Viterbo's homogenization operator.

\subsection{What happens in the (weakly) exact case?}
\label{adapcob}
If one chooses to work with weakly exact Lagrangians one can obtain the results in Section \ref{MainResults} in a slightly different form. For the convenience of the reader we here point out these changes.

\subsubsection{The case $\mu|_{\pi_2(M,L)}\equiv 0$}
One alternative construction of spectral invariants which is relevant for our purposes was carried out by Leclercq \cite{Leclercq08} for closed Lagrangians $L\subset (M,\omega)$ verifying 
\begin{equation}
\label{eqweak}
\omega|_{\pi_2(M,L)}\equiv 0 \quad \& \quad \mu|_{\pi_2(M,L)}\equiv 0.
\end{equation}
Note that the existence of such a Lagrangian in $(M,\omega)$ implies that $M$ is symplectically aspherical in the sense that $\omega|_{\pi_2(M)}\equiv 0$ and $c_1(TM)|_{\pi_2(M)}\equiv 0$. We denote by $\mathcal{L}_0$ the space of all closed Lagrangians $L\subset (M,\omega)$ satisfying (\ref{eqweak}). Note that for any $L\in \mathcal{L}_0$, $QH_*(L)$ reduces to $H_*(L)=H_*(L;\Z_2 )$. As already mentioned this setting is covered by our results in Section \ref{MainResults}. However, it is also possible to recover some of our results using Leclercq's spectral invariants which satisfy particularly nice properties. The spectral invariant function 
\[
c(\cdot \ ;L,\cdot ):H_*(L)\times \mathcal{H}(L)\to \R \cup \{-\infty \},
\] 
constructed by Leclercq in \cite{Leclercq08}, is associated to a fixed $L\in \mathcal{L}_0$. Among other properties he showed that $c(\alpha;L, L')=-\infty $ if and only if $\alpha=0\in H_*(L)$ and, for every pair $(\alpha, L')\in H_*(L)\times \mathcal{H}(L)$ for which $\alpha \neq 0$, one has $0\leq c(\alpha;L,L')\leq d_H(L,L')$. For $L\in \mathcal{L}_0$ the results in Section \ref{secLagmet} continue to be true, \emph{mutatis mutandis}, if one replaces $l_L$ by $c(\cdot \ ; L,\cdot)$. More precisely, when replacing $l_L$ by $c(\cdot \ ;L,\cdot )$, $QH_*(L)$ is replaced by $H_*(L)$, $\widetilde{\Ham}(M,\omega)$ is replaced by $\mathcal{H}(L)$ and the statements hold for elementary Lagrangian cobordisms $V$ verifying $\tilde{\omega}|_{\pi_2(\tilde{M},V)}\equiv 0\equiv \mu|_{\pi_2(\tilde{M},V)}$. One could of course also study the asymptotic version of $c(\cdot \ ; L,\cdot )$. However, it is not clear to us that this quantity contains information about Lagrangian cobordisms. 

\subsubsection{The case $\mu|_{\pi_2(M,L)}\neq 0$}
\label{weakex}
Recall that a Lagrangian $L\subset (M,\omega)$ is said to be weakly exact if $\omega|_{\pi_2(M,L)}\equiv 0$. We denote by $\mathcal{L}_{we}(M,\omega)$ the space of all closed and weakly exact Lagrangian submanifolds in $(M,\omega)$. In case $(M,\omega=d\lambda )$ is exact\footnote{Recall that $(M,\omega)$ is said to be \emph{exact} if $\omega=d\lambda$ for some 1-form $\lambda$ on $M$. In this case a Lagrangian $L\subset (M,d\lambda)$ is said to be \emph{exact} (with respect to $\lambda$) if $\lambda|_L=df$ for some $f\in C^{\infty}(L)$.} the exact Lagrangians are special cases of weakly exact Lagrangians. The version of spectral invariants developed in \cite{LeclercqZapolsky15} was initially constructed in the exact setting for the particular case of the zero-section in a cotangent bundle by Oh \cite{Oh97}, \cite{Oh99}. The parts of Oh's scheme which are needed for our results can also be carried out for Lagrangians in $\mathcal{L}_{we}(M,\omega)$ (see \cite{LeclercqZapolsky15}, \cite{Zapolsky13}). In this setting $HF(L)$ does not necessarily carry a $\Z$-grading but is still isomorphic to $H(L)=\oplus_{k=0}^nH_k(L;\Z_2)$. Zapolsky \cite{Zapolsky13} showed that, for $L\in \mathcal{L}_{we}$, $l_L$ in fact descends to $\Ham(M,d\lambda)$:
\[
l_L:H(L)\times \Ham(M,d\lambda)\to \R \cup \{-\infty \}.
\]
Therefore, one recovers all results from Section \ref{Lagspecinv} with the one difference that $\widetilde{\Ham}(M,\omega)$ can be replaced by $\Ham(M,d\lambda)$ throughout, given that one also restricts to looking at Lagrangian cobordisms $V$ satisfying $\tilde{\omega}|_{\pi_2(\tilde{M},V)}\equiv 0$. The same goes for the results in Section \ref{secLagmet}. For additional properties of $\sigma_L$ for $L\in \mathcal{L}_{we}$ we refer to \cite{ZapolskyMonznerVichery12}, where the case of a zero-section in a cotangent bundle is studied in detail. 
\subsection*{Acknowledgment}
The work presented here is carried out in the framework of my PhD at the ETH Z{\"u}rich. I am grateful to my advisors Paul Biran and Will J. Merry for all the helpful discussions. Especially I want to thank Paul for encouraging me to think independently about Lagrangian cobordisms. I also want to thank Frol Zapolsky for generously sharing his ideas on spectral invariants, Luis Haug for patiently helping me understand his work \cite{Haug15} and Egor Shelukhin as well as R{\'e}mi Leclercq for helping me improve the exposition of my results. Last but certainly not least I am indebted to the anonymous referee whose careful reading and generous advise significantly improved the quality of the paper.  

\section{Preliminaries on Lagrangian spectral invariants}
\label{SecLagInv}
Fix $L\in \mathcal{L}_{\tau}$. Given a pair $(H,J)$, where $H\in C^{\infty}([0,1]\times M)$ is a (time-dependent) Hamiltonian satisfying $\phi_H^1(L) \pitchfork L$ and $J=\{J_t\}_{t\in [0,1]}$ is a generic smooth path of $\omega$-compatible almost complex structures, one can construct the Floer homology group $HF_*(H,J:L)$. Recall that $HF_*(H,J:L)$ can be thought of as the Morse homology of the action functional $\mathcal{A}_{H:L}$. Here we view $\mathcal{A}_{H:L}$ as being defined on the space $\overline{\Omega}_L$ consisting of equivalence classes of pairs $\widetilde{\gamma}=[\gamma,\widehat{\gamma}]$ where $\gamma:([0,1],\{0,1\})\to (M,L)$ and $\widehat{\gamma}:(\dot{D}^2, \partial \dot{D}^2 )\to (M,L)$ is a capping of $\gamma$ ($\dot{D}^2=D^2\backslash \{1\} \subset \C$ denotes the punctured unit disc). The equivalence relation is given by identifying cappings of equal symplectic area. Following \cite{Zapolsky15} we use the convention\footnote{Note that $\mathcal{A}_{H:L}$ is defined \emph{absolutely} here. In other words, since the definition of $HF_*(H,J:L)$ in \cite{LeclercqZapolsky15} and \cite{Zapolsky15} does not require the choice of a base point in $\Omega_L$ there is no need to normalize spectral invariants. This will be important below.}
\begin{equation*}
\mathcal{A}_{H:L}(\widetilde{\gamma}=[\gamma, \widehat{\gamma}])=\int_0^1 H_t(\gamma(t)) \ dt - \int \widehat{\gamma}^*\omega, \quad \widetilde{\gamma}\in \overline{\Omega}_L.	
\end{equation*}
Lagrangian Floer homology was first developed by Floer \cite{Floer88} and later developments were carried out by Oh \cite{Oh93}, \cite{Oh931}. Today Lagrangian Floer theory is a well-documented theory and some of the standard references to which we refer for further details are \cite{Seidel08}, \cite{Oh151} and \cite{Oh152}. Here and throughout the paper we will follow the conventions and notation appearing in \cite{Zapolsky15}, to which we also refer the interested reader. Assuming $L\in \mathcal{L}^*_{\tau}$ one can use $HF_*(H,J:L)$ to extract so-called \emph{spectral invariants}. This idea was recently developed in the monotone setting by Leclercq and Zapolsky \cite{LeclercqZapolsky15}. Leclercq and Zapolsky constructed a \emph{Lagrangian spectral invariant} function
\begin{equation}
\label{cob20}
l_L:QH_*(L)\times \widetilde{\Ham}(M,\omega)\to \R \cup \{-\infty\},
\end{equation}
satisfying $l_L(\alpha,\phi)=-\infty$ if and only if $\alpha=0\in QH_*(L)$. This function is defined by "mimicking" classical critical point theory as follows. $HF_*(H,J:L)$ is the homology of the Floer chain complex $(CF_*(H,J:L),d)$ where $CF_*(H,J:L)$ is the $\Z_2$-vector space generated by critical points of $\mathcal{A}_{H:L}$ and $d$ is defined by "counting finite energy Floer trajectories". Given $a\in \R$ we denote by $CF^a_*(H,J:L)\subset CF_*(H,J:L)$ the subspace generated by those $\crit(\mathcal{A}_{H:L})$-points whose action is $<a$. Floer-trajectories can be interpreted as negative gradient flow lines for $\mathcal{A}_{H:L}$, so $d$ restricts to a differential on $CF^a_*(H,J:L)$. We denote by
\[
\iota^a:CF^a_*(H,J:L) \hookrightarrow CF_*(H,J:L)
\]
the inclusion and by $\iota^a_*:HF^a_*(H,J:L) \rightarrow HF_*(H,J:L)$ the map induced on homology. Identifying all the groups $HF_*(H,J:L)$ for different choices of data $(H,J)$ we obtain the Floer homology ring of $L$, $HF_*(L)$. After choosing a quantum datum for $L$, $QH_*(L)$ is well-defined and ring-isomorphic to $HF_*(L)$ via a PSS-type isomorphism
\[
\text{PSS}:QH_*(L)\stackrel{\cong}{\longrightarrow} HF_*(L).
\]
Given $\alpha \in QH_*(L)$ and a Floer datum $(H,J)$ Leclercq and Zapolsky define
\[
l_L(\alpha,H,J):=\inf\{a\in \R \ |\ \text{PSS}(\alpha)\in \image(\iota_*^a)\subset HF_*(H,J:L)\}.
\]
They then further show that $l_L(\alpha,H,J)$ is independent of $J$ and that $l_L$ descends to a function (\ref{cob20}) satisfying many additional properties \cite{LeclercqZapolsky15}.

\section{Preliminaries on Lagrangian cobordism}
\label{secLagCob}
Recently Biran and Cornea introduced several new methods for studying Lagrangian submanifolds via Lagrangian cobordisms \cite{BiranCornea13}, \cite{BiranCornea14}. Here we follow their work. Recall that $\pi :\tilde{M}=\R^2 \times M\to \R^2$ denote the canonical projection. For subsets $V\subset \tilde{M}$ and $U\subset \R^2$ we write $V|_{U}=V\cap \pi^{-1}(U)$.
\begin{definition}
We say that two families $(L_i)_{i=0}^{k_-}$ and $(L'_j)_{j=0}^{k_+}$ of closed connected Lagrangian submanifolds of $(M,\omega)$ are \emph{Lagrangian cobordant} if for some $R>0$ there exists a smooth compact Lagrangian submanifold $V\subset ([-R,R]\times \R \times M, \omega_{\R^2}\oplus \omega )$ with boundary $\partial V =V\cap (\{\pm R\}\times \R \times M)$ satisfying the condition that for some $\epsilon >0$ we have
\begin{align}
\label{cob21}
V|_{[-R,-R+\epsilon)\times \R}&=\bigsqcup_{i=0}^{k_-}([-R,-R+\epsilon)\times \{i\})\times L_i \\
\label{cob22}
V|_{(R-\epsilon,R]\times \R}&=\bigsqcup_{j=0}^{k_+}((R-\epsilon,R]\times \{j\})\times L'_j.
\end{align}
In particular $V$ defines a smooth compact cobordism $(V,\bigsqcup_{i=0}^{k_-}L_i,\bigsqcup_{j=0}^{k_+}L'_j)$. We write $V:(L'_j)_j \rightsquigarrow (L_i)_i$.

\end{definition}  
Our notation will not distinguish between a Lagrangian cobordism and its obvious horizontal $\R$-extension. This extension is a \emph{Lagrangian with cylindrical ends}. More generally we have
\begin{definition}[\cite{BiranCornea13}]
A \emph{Lagrangian with cylindrical ends} is a Lagrangian submanifold $V\subset (\tilde{M},\tilde{\omega})$ without boundary satisfying the conditions that $V|_{[a,b]\times \R}$ is compact for all $a<b$ and that there exists $R>0$ such that 
\begin{align*}
V|_{(-\infty,-R]\times \R}&=\bigsqcup_{i=0}^{k_-}((-\infty,-R]\times \{a^-_i\})\times L_i \\
V|_{[R,\infty)\times \R}&=\bigsqcup_{j=0}^{k_+}([R,\infty)\times \{a^+_j\})\times L'_j
\end{align*}
for Lagrangians $L_i,L'_j \subset (M,\omega)$ and constants $a^-_i,a^+_j\in \R$ verifying $a^-_i\neq a^-_{i'}$ for $i\neq i'$ and $a^+_j\neq a^+_{j'}$ for $j\neq j'$.
\end{definition}
We will be interested in specific Lagrangian cobordisms and Lagrangians with cylindrical ends which allow us to compare Floer-theoretic invariants of the ends.
\begin{definition}
Given two families $(L_i)_{i=0}^{k_-},(L'_j)_{j=0}^{k_+}\subset \mathcal{L}_{\tau}$ we say that a Lagrangian cobordism $V:(L'_j)_j \rightsquigarrow (L_i)_i$ is \emph{admissible} if $V\subset (\tilde{M},\tilde{\omega})$ is itself a monotone, connected Lagrangian submanifold with monotonicity constant $\tau_V=\tau$ and minimal Maslov number $N_V \geq 2$. We say that $V$ is an \emph{elementary Lagrangian cobordism} if $V$ is admissible and satisfies $k_+=k_-=0$, i.e. if there is only one positive and one negative end.
\end{definition} 
For examples of Lagrangian cobordisms we refer to \cite{Haug15}, \cite{BiranCornea13} and \cite{Chekanov97}.
\begin{remark}
Note that "being cobordant by an elementary Lagrangian cobordism" is an equivalence relation on $\mathcal{L}_{\tau}$.
\end{remark}
\begin{remark}
\label{cobrem1}
If $L,L'\in \mathcal{L}_{\tau}$ are in the same elementary Lagrangian cobordism class then the Floer homology group $HF(L,L')$ with coefficients in the universal Novikov ring over the base ring $\Z_2$ is well-defined \cite{BiranCornea13}. If $QH_*(L)\neq 0$ then results from \cite{BiranCornea13} imply that $HF(L,L')\neq 0$. In particular Corollary \ref{cobcor191919} follows.
\end{remark}
Quantum (and Floer) homology for Lagrangians with cylindrical ends was introduced by Biran and Cornea \cite{BiranCornea13}, \cite{BiranCornea14} and further studied by Singer \cite{Singer15}. Since action estimates are crucial for our intentions we will make some small adaptions in the construction of Lagrangian Floer homology from \cite{BiranCornea13} to make it suit our purposes.

\section{Proofs of results}
\label{Lagsecproof}
Here we develop the theory needed to prove our results. Most of our results are in fact consequences of Theorem \ref{lemcob1} whose proof we postpone until the end. For the proof of Proposition \ref{propref} it will be convenient to view $l_L$ as a function 
\[
l_L:QH_*(L)\times C^{\infty}_c([0,1]\times M)\to \R \cup \{-\infty\},
\]
so that we don't have to worry about normalizing our Hamiltonians \cite{LeclercqZapolsky15}. Given $L\in \mathcal{L}_{\tau}$ we will denote by $U=U(L)\subset M$ a Darboux-Weinstein neighborhood of $L\subset M$. In particular we have a neighborhood $W=W(L)\subset T^*L$ of $L\subset T^*L$ and a symplectic identification $U\approx W$ which restricts to the identity on $L$ \cite{McDuffSalamon98}. For the proof of Proposition \ref{propref} we will need
\begin{lemma}
	\label{lemref1}
	Fix $L\in \mathcal{L}^*_{\tau}$. Denote $b:U\to L$ the restriction of the base-point map $T^*L\to L$ to $W\approx U$. Let $h\in C^{\infty}(L)$ be a Morse function such that $\max_L|h|< \tfrac{\tau_L N_L}{2}$ and $\Graph(dh)\subset Y$, where $Y$ is a precompact and fiber-wise convex neighborhood of the 0-section in $W\approx U$. Define $H\in C^{\infty}_c(M)$ by $H:=\varphi b^*h$, where $\varphi \in C^{\infty}_c(U;[0,1])$ is a cutoff satisfying $\varphi|_Y\equiv 1$. Then there exists $q\in \crit_n(h)$ such that\footnote{Here $\crit_n(h)$ denotes the critical points of $h$ whose Morse index equals $n$.}
	\[
	l^+_L(H)=h(q).
	\]
\end{lemma}

\begin{proof}[Proof of Proposition \ref{propref}]
	Let $L,L'\in \mathcal{L}^*_{\tau}$ with $L\neq L'$. Denote by $b:U(L) \to L$ the restriction of the base-point map $T^*L\to L$ to $W(L)\approx U(L)$. Choose $q\in L \backslash L'$ and a Morse function $f\in C^{\infty}(L)$ such that $\crit_n(f)=\{q\}$. Fix a Morse chart $B\subset L\backslash L'$ at $q$ and a bump function $h\in C_c^{\infty}(B)$ attaining its unique maximum at $q$ with $0<h(q)<\tfrac{\tau_LN_L}{2}$. By perhaps rescaling $h$ we may assume that $\Graph(dh|_B)\cap (L'\cap U(L))=\emptyset$ and choose a cutoff $\varphi \in C^{\infty}_c(U(L))$ as in Lemma \ref{lemref1} such that $\varphi|_{b^{-1}(B)\cap L'}=0$. Consider for small $\epsilon \geq 0$ the autonomous Hamiltonian $H^{\epsilon}:=\varphi b^*(h+\epsilon f)\in C^{\infty}_c(M)$. Since $H^{0}|_{L'}\equiv 0$ the Lagrangian control property from \cite{LeclercqZapolsky15} implies that $l^+_{L'}(H^{0})=0$. Moreover, for all small $\epsilon >0$ we have
	\[
	l^+_L(H^{\epsilon})=h(q)+\epsilon f(q)
	\]
	by the lemma and $H^{\epsilon} \stackrel{\epsilon \to 0}{\longrightarrow} H^{0}$ uniformly. Thus, continuity of $l_L^{+}$ implies $l^+_L(H^{0})=h(q)$. Hence,
	\[
	d_s(L,L')\geq |l_L^+(\phi_{H^{0}})-l_{L'}^+(\phi_{H^{0}})|=|l_L^+(H^{0})-l_{L'}^+(H^{0})|=h(q) >0.
	\]
\end{proof}
\begin{proof}[Proof of Lemma \ref{lemref1}]
	By the spectrality property of $l_L$ \cite{LeclercqZapolsky15} and the fact that $[L]\in QH_n(L)$ we know that $l_L^+(H)= \mathcal{A}_{H:L}([\gamma, \widehat{\gamma}])$ for some $[\gamma, \widehat{\gamma}]\in \crit(\mathcal{A}_{H:L})$ whose Conley-Zehnder index equals $n$. By construction of $H$ we can identify $[\gamma, \widehat{\gamma}]\approx [q,\widehat{q}]$ where $q\in \crit(h)$ and $\widehat{q}$ is a topological disc in $M$ with boundary on $L$. The Conley-Zehnder index of $[q,\widehat{q}]$ equals $|q|_h-\mu(\widehat{q})$, where $|q|_h$ denotes the Morse index of $q$ and $\mu$ denotes the Maslov index \cite{Zapolsky15}. We claim that we must have $|q|_h=n$. To see this, assume for contradiction that $|q|_h<n$. Then $\mu(\widehat{q})=|q|_h-n<0$ and thus
	\begin{align*}
		\mathcal{A}_{H:L}([\gamma, \widehat{\gamma}])&=\int_0^1H(\gamma(t))dt-\int \widehat{\gamma}^*\omega =h(q)-\omega(\widehat{q}) \\
		&=h(q)-\tau_L\mu(\widehat{q})\geq -\max_{L}|h|+\tau_LN_L> \frac{\tau_LN_L}{2}. 
	\end{align*}
	But by the continuity property of $l^+_L$ we also have 
	\[
	|\mathcal{A}_{H:L}([\gamma, \widehat{\gamma}])|=|l_L^+(H)|\leq \max_M|H|=\max_L|h|<\frac{\tau_LN_L}{2},
	\] 
	which is a contradiction. This shows that $|q|_h=n$ and therefore $\mu(\widehat{q})=|q|_h-n=0$. It follows that 
	\[
	l_L^+(H)= \mathcal{A}_{H:L}([\gamma, \widehat{\gamma}])=h(q)-\tau_L \mu(\widehat{q})=h(q).
	\]
\end{proof}

\begin{remark}
	In the above proof we used the $\Z$-grading on $HF_*(L)$. If $L\in \mathcal{L}_{we}(M,\omega)$ with $\mu|_{\pi_2(M,L)}\neq 0$ then $HF(L)$ does not necessarily carry a $\Z$-grading (see Section \ref{weakex}). However, Proposition \ref{propref} continues to hold true also in this setting. We will not need this and therefore not carry out the proof. The basic idea is that, if the condition $\max_L|h|< \tfrac{\tau_L N_L}{2}$ in the statement of Lemma \ref{lemref1} is replaced by the condition that $h$ be $C^2$-small, then the Floer chain complex $CF(H,J:L)$ of the weakly exact Lagrangian $L$ reduces to the Floer chain complex $CF(H|_U,J|_U:L)$ of $L$ viewed as a the 0-section in $W\approx U$ (see \cite{Oh96}). But this chain complex carries a $\Z$-grading, simply given by the index of the critical points of $h$, so the above argument can be carried out.
\end{remark}

\begin{proof}[Proof of Theorem \ref{thmcob1}]	
Note first that the existence of $\Phi_V$ guarantees that any $L'\in \mathcal{L}_{\tau}$ in the same elementary Lagrangian cobordism class as $L$ is in fact an element of $\mathcal{L}^*_{\tau}$, so indeed $\sigma_{L'}:\widetilde{\Ham}(M,\omega)\to \R$ is well-defined. Moreover $L'$ and $L$ are in the same class if and only if $d_c(L,L')<\infty$. Assuming this is the case Corollary \ref{corcob1} gives
\begin{equation*}
|\sigma_L(\phi)-\sigma_{L'}(\phi)|= \lim_{k\to \infty}\frac{|l_L^+(\phi^k)-l_{L'}^+(\phi^k)|}{k}\leq \lim_{k\to \infty}\frac{d_c(L,L')}{k}=0
\end{equation*}
for all $\phi \in \widetilde{\Ham}(M,\omega)$.
\end{proof}

\begin{proof}[Proof of Corollary \ref{cobcor100}]
Note first that $\psi(L)\in \mathcal{L}_{\tau}$. Recall from \cite{BiranCornea07} and \cite{LeclercqZapolsky15} that any $\psi\in \Symp(M,\omega)$ induces an isomorphism
\[
\psi_*:QH_*(L) \to QH_*(\psi(L)).
\]
Clearly $\psi_*$ maps $[L]$ to $[\psi(L)]$. In particular it follows from the symplectic invariance property of Lagrangian spectral invariants \cite{LeclercqZapolsky15} that 
\begin{equation*}
l_L^+(\phi)=l_{\psi(L)}^+(\psi \phi \psi^{-1}) \quad \forall \ \phi \in \widetilde{\Ham}(M,\omega).
\end{equation*}
Assuming the existence of an elementary Lagrangian cobordism $V:L\rightsquigarrow \psi(L)$ it therefore follows from Theorem \ref{thmcob1} that
\begin{align*}
\sigma_L(\phi)&=\lim_{k\to \infty}\frac{l_L^+(\phi^k)}{k}=\lim_{k\to \infty}\frac{l_{\psi(L)}^+(\psi \phi^k \psi^{-1})}{k}\\
&=\lim_{k\to \infty}\frac{l_{\psi(L)}^+((\psi \phi \psi^{-1})^k)}{k}=\sigma_{\psi(L)}(\psi \phi \psi^{-1})=\sigma_{L}(\psi \phi \psi^{-1}).
\end{align*}
\end{proof}


\subsection{Floer homology, PSS and spectral invariants for Lagrangians with cylindrical ends}

Throughout this section we consider a connected monotone Lagrangian submanifold $V\subset (\tilde{M},\tilde{\omega})$ with cylindrical ends and minimal Maslov number $N_V \geq 2$. We denote by $(L_i)_{i=0}^{k_-}$ the family of Lagrangians in $(M,\omega )$ corresponding to negative ends of $V$ and by $(L'_j)_{j=0}^{k_+}$ the family of Lagrangians in $(M,\omega )$ corresponding to positive ends of $V$.

\subsubsection{Floer homology with Hamiltonian perturbations for Lagrangians with cylindrical ends} 
\label{HFcylends}
Our reference for Lagrangian Floer homology is \cite{Zapolsky15} and we adopt the conventions used there. Since we work in the setting of Lagrangians with cylindrical ends we will apply the machinery developed in \cite{BiranCornea13} and \cite{BiranCornea14} to deal with compactness issues. Due to the fact that we use many different references we will here point out how to combine the different approaches.

Here, following \cite{BiranCornea14}, Floer homology will be based on the choice of a class of \emph{perturbation functions} $h\in C^{\infty}(\R^2)$. Our requirements of $h$ will differ slightly from those in \cite{BiranCornea14}. We therefore point out the specific conditions which $h$ needs to satisfy. Fix a number $R>0$ such that $V$ is cylindrical outside $[-R,R]^2$,
\begin{align}
\label{cob6}
V|_{\R^2 \backslash [-R,R]^2}=& \left(\bigsqcup_{i=0}^{k_-}(-\infty,-R]\times \{a^-_i\}\times L_i\right)\cup \left( \bigsqcup_{j=0}^{k_+}[R,\infty)\times \{a^+_j\}\times L_j'\right).
\end{align}
We will require the following of $h$.
\begin{enumerate}[(ii)]
\item
Fix $\epsilon >0$ so small that all the sets $V_j^+=[R,\infty)\times [a_j^+-\epsilon,a_j^++\epsilon]$ and $V_j^-=(-\infty,-R]\times [a_j^--\epsilon,a_j^++\epsilon]$ are pairwise disjoint. We require that the support of $h$ be contained in the union of these and $[-C,C]^2$ where $C:=R+1$.
\item
The Hamiltonian isotopy $\phi_h^t$ associated to $h$ exists for all $t\in \R$.
\item
The restriction of $h$ to each of the sets $T_j^+=[C,\infty)\times [a_j^+-\tfrac{\epsilon}{2},a_j^++\tfrac{\epsilon}{2}]$ and $T_j^-=(-\infty,-C]\times [a_j^--\tfrac{\epsilon}{2},a_j^++\tfrac{\epsilon}{2}]$ takes the form
\begin{equation}
\label{cob5}
h(x,y)=\alpha_j^{\pm}x+\beta_j^{\pm}
\end{equation}
where each $\alpha_j^{\pm}\in \R \backslash \{0\}$ has absolute value so small that
\begin{align*}
\phi_h^t([C,\infty)\times \{a_j^+\})&\subset T_j^{+} \quad \forall \ t\in [-1,1]
\end{align*}
and
\begin{align*}
\phi_h^t((-\infty, -C]\times \{a_j^-\})&\subset T_j^{-} \quad \forall \ t\in [-1,1].
\end{align*}
\item
$\phi_h^t([-C,C]^2)=[-C,C]^2$ for all $t\in [-1,1]$.
\end{enumerate}

It is easy to verify the existence of such an $h$ and having fixed one we denote by $\mathfrak{h}$ the corresponding class of perturbation functions. This class is defined as follows: $h' \in C^{\infty}(\R^2)$ is an element of $\mathfrak{h}$ if and only if it satisfies (i)-(iv) and $h=h'$ outside $[-C,C]^2$.

Given a fixed class of perturbation functions $\mathfrak{h}$ we now specify the requirements for the data going into the definition of the Floer chain complexes we want to consider.
\begin{enumerate}[(ii)]
	\item	$\tilde{\mathcal{H}}_{\mathfrak{h}}$ denotes the space of all Hamiltonians $\tilde{H}\in C^{\infty}([0,1]\times \tilde{M})$ satisfying the condition that there is a compact subset $Y\subset (-C,C)^2$ (depending on $\tilde{H}$) such that
	\begin{equation}
	\label{cob1}
	\tilde{H}_t(z,p)=h(z)+H_t(p) \quad \forall \ (t,z,p)\in [0,1]\times (\R^2\backslash Y)\times M,
	\end{equation}
	for some Hamiltonian $H\in C_c^{\infty}([0,1]\times M)$ and some $h\in \mathfrak{h}$.
	\item
	$\tilde{\mathcal{J}}_{\mathfrak{h}}$ denotes the space of time dependent $\tilde{\omega}$-compatible almost complex structures $\tilde{J}=\{\tilde{J}_t\}_{t\in [0,1]}$ on $\tilde{M}$ satisfying the additional condition that the canonical projection $\pi: \tilde{M} \to \R^2$ restricts to a $(\tilde{J}_t,(\phi_h^t)_*i)$-holomorphic map on $(\R^2 \backslash [-C,C]^2)\times M$ for all $t\in [0,1]$ . Here $i$ denotes the canonical complex structure on $\C \approx \R^2$ and $h$ is some element of $\mathfrak{h}$.
\end{enumerate}

Note that $\tilde{\mathcal{H}}_{\mathfrak{h}}$ is a convex space. Given a non-degenerate $\tilde{H}\in \tilde{\mathcal{H}}_{\mathfrak{h}}$, in the sense that $\phi_{\tilde{H}}^1(V) \pitchfork V$, and a generic $\tilde{J}\in \tilde{\mathcal{J}}_{\mathfrak{h}}$ we want to consider the Floer chain complex 
\[
(CF_*(\tilde{H},\tilde{J}:V),d ),
\]
defined in \cite{Zapolsky15}. Due to our non-compact setting we need to verify that all finite energy Floer trajectories, i.e. finite energy solutions $u:\R \times [0,1]\to \tilde{M}$ of Floer's equation 
\[
\left\{
\begin{array}{l}
\partial_su +\tilde{J}_t(u)(\partial_tu -X_{\tilde{H}_t}(u))=0 \\
u(\R \times \{0,1\})\subset V,
\end{array}
\right.
\] 
stay in a compact set. The next proposition ensures that this is the case.
\begin{proposition}
\label{propcob1}
Let $\tilde{H}\in \tilde{\mathcal{H}}_{\mathfrak{h}}$ be non-degenerate and let $\tilde{J}\in \tilde{\mathcal{J}}_{\mathfrak{h}}$. Then all finite energy solutions of Floer's equation are contained in $[-C,C]^2\times M$. As a consequence the pair $(\tilde{H},\tilde{J})$ is regular for generic $\tilde{J}\in \tilde{\mathcal{J}}_{\mathfrak{h}}$ in the sense that $(CF_*(\tilde{H},\tilde{J}:V),d )$ is a well-defined chain complex. Moreover, for every two regular Floer data $(\tilde{H}^-,\tilde{J}^-),(\tilde{H}^+,\tilde{J}^+)\in \tilde{\mathcal{H}}_{\mathfrak{h}} \times \tilde{\mathcal{J}}_{\mathfrak{h}}$ and every regular homotopy of Floer data from $(\tilde{H}^-,\tilde{J}^-)$ to $(\tilde{H}^+,\tilde{J}^+)$, there is a continuation chain map 
\begin{equation*}
(CF_*(\tilde{H}^-,\tilde{J}^-:V),d )\to (CF_*(\tilde{H}^+,\tilde{J}^+:V),d ),
\end{equation*}
which induces an isomorphisms on homology. This isomorphism is canonical in the sense that it is independent of the choice of regular homotopy of Floer data. 
\end{proposition}
\begin{proof}
	This follows immediately from the compactness and transversality arguments carried out in \cite{BiranCornea13} and \cite{BiranCornea14}. In fact the compactness argument runs analogously to the one carried out in the proof of Proposition \ref{propcob2} below.
\end{proof}
As usual we will identify all Floer homology groups via the canonical isomorphisms induced by continuation maps. In this way we obtain an abstract Floer homology group which we denote by $HF_*(V,\mathfrak{h})$. 

As explained in Section \ref{SecLagInv} the main structure needed to extract spectral invariants from homology groups is an $\R$-filtration. Given $a \in \R$ and a regular Floer datum $(\tilde{H},\tilde{J})\in \tilde{\mathcal{H}}_{\mathfrak{h}}\times \tilde{\mathcal{J}}_{\mathfrak{h}}$, we denote by $(CF_*^{a}(\tilde{H},\tilde{J}:V),d)$ the Floer chain complex generated by critical points of the action functional $\mathcal{A}_{\tilde{H}:V}$ whose action is $<a$. This is a well-defined chain complex because the Floer differential $d$ is action decreasing. We denote by $HF_*^{a}(\tilde{H},\tilde{J}:V)$ its homology and by $\iota^{a}_*:HF_*^{a}(\tilde{H},\tilde{J}:V)\to HF_*(\tilde{H},\tilde{J}:V)$ the map induced by the inclusion $\iota^{a}:(CF_*^{a}(\tilde{H},\tilde{J}:V),d)\hookrightarrow (CF_*(\tilde{H},\tilde{J}:V),d)$.


\subsubsection{The PSS isomorphism}
Suppose we are given a Lagrangian $V\subset (\tilde{M},\tilde{\omega})$ with cylindrical ends as above together with a regular quantum datum $\mathcal{D}=(\tilde{f},\tilde{\rho},\tilde{J}')$ \emph{adapted to the exit region} $S=\partial V$ in the sense of Section 3 in \cite{Singer15}. Here $(\tilde{f},\tilde{\rho})$ denotes a Morse-Smale pair on $V$ satisfying additional conditions as in \cite{Singer15}. In particular $\tilde{f}$ is required to be split on
\[
\left(\bigsqcup_{i=0}^{k_-}(-\infty,-R+\delta]\times \{a^-_i\}\times L_i \right)\cup \left( \bigsqcup_{j=0}^{k_+}[R-\delta,\infty)\times \{a^+_j\}\times L'_j\right)
\]
for some small $\delta > 0$ and $-\nabla^{\tilde{\rho}}\tilde{f}$ is required to point outwards along $\partial V|_{[-R,R]^2}$. Also, $\tilde{J}'$ denotes a generic almost complex structure on $\tilde{M}$ satisfying the condition that $\pi:\tilde{M}\to \R^2$ restricts to a $(\tilde{J}',i)$-holomorphic function on $(\R^2 \backslash [-R+\delta,R-\delta]^2) \times M$. As showed in \cite{Singer15} and \cite{BiranCornea13} the quantum chain complex $(QC_*(\mathcal{D}:V,\partial V),d)$ is then an honest chain complex whose homology $QH_*(\mathcal{D}:V,\partial V)$ is independent of the choice of regular quantum datum $\mathcal{D}$. 

We now fix a choice of perturbation function $h$ for $V$ and require it satisfy the following condition, which is identical to the one used in Section 5.2 of \cite{BiranCornea13}.
\begin{equation}
	\label{cob7}
	\begin{array}{c}
	 \text{For every $j\in \{1,\ldots ,k_{\pm}\}$ the constant $\alpha_j^{\pm}$ in} \\ \text{(\ref{cob5}) is required to satisfy $\pm \alpha_j^{\pm}<0$.}  
	\end{array}
\end{equation}
Denote by $\mathfrak{h}$ the corresponding class of perturbation functions. It was discovered in \cite{BiranCornea13} (see also Remark 3.5.1. in \cite{BiranCornea14}) that this specific choice of class implies that there is a PSS-type isomorphism $QH_*(V,\partial V) \cong HF_*(V,\mathfrak{h})$. Fixing a regular Floer datum $(\tilde{H},\tilde{J})\in \tilde{\mathcal{H}}_{\mathfrak{h}}\times \tilde{\mathcal{J}}_{\mathfrak{h}}$ we will now point out how this isomorphism adapts to our setup. More precisely, we will define chain maps
\begin{align}
\label{cob2}
\PSS_+:QC_*(\mathcal{D}: V,\partial V)\to CF_*(\tilde{H},\tilde{J}:V)\\
\label{cob3}
\PSS_-:CF_*(\tilde{H},\tilde{J}:V)\to QC_*(\mathcal{D}: V,\partial V)
\end{align}
which, at the level of homology, are inverse to each other and induce a canonical isomorphism $QH_*(V,\partial V)\cong HF_*(V,\mathfrak{h})$. In the standard case of closed monotone Lagrangians of closed symplectic manifolds this was carried out in \cite{BiranCornea07}. Moreover, the construction is described in great detail in \cite{Zapolsky15}. 

We first introduce some notation. Define $Z:=\R \times [0,1]$ and $Z_{\pm}:=\{(s,t)\in Z\ | \ \pm s \geq 0\}$, viewed as subsets of $\R^2 \approx \C$. We will think of $D^2=\{z\in \C \ |\ |z|\leq 1\}$ as a Riemann surface with boundary equipped with the conformal structure it inherits from $\C$. Define also $D_{\pm}:=D^2 \backslash \{\pm 1\}$ where we view $\pm 1$ as a positive $(+)$, repectively negative $(-)$, boundary puncture in the sense of \cite{Zapolsky15} (see also \cite{Seidel08}) and equip the punctures with the standard strip-like ends $\epsilon_{\pm}:Z_{\pm}\to D_{\pm}$ given by
\[
\epsilon_{\pm}(z)=\frac{e^{\pi z}-i}{e^{\pi z}+i}, \quad z\in Z_{\pm}.
\]
Choose once and for all two functions $a_{\pm}\in C^{\infty}(D_{\pm},[0,1])$ satisfying the following conditions:
\begin{enumerate}[(ii)]
\item
$a_{\pm}(z)=0$ whenever $z \notin \image(\epsilon_{\pm})$ or $z=\epsilon_{\pm}(s,t)$ for $t=0$ and/or $\pm s\leq 1$.
\item
$a_{\pm}(\epsilon_{\pm}(s,t))=t$ whenever $\pm s\geq 2$.
\item
$\pm \partial_s(a_{\pm}\circ \epsilon_{\pm})(s,1)>0$ whenever $1<\pm s<2$.
\end{enumerate} 
Following \cite{BiranCornea14} we consider now a specific type of perturbation data $(\tilde{K}^{\pm},\tilde{I}^{\pm})$ on $D_{\pm}$, compatible with the Floer data $(\tilde{H},\tilde{J})$. That is, we will consider pairs $(\tilde{K}^{\pm},\tilde{I}^{\pm})$ where $\tilde{K}^{\pm}\in \Omega^1(D_{\pm},C^{\infty}(\tilde{M}))$ is a 1-form on $D_{\pm}$ with values in $C^{\infty}(\tilde{M})$ and $(\tilde{I}^{\pm}_z)_{z\in D_{\pm}}$ is a family of $\tilde{\omega}$-compatible almost complex structures on $\tilde{M}$. The specific requirements we make are as follows.
\begin{enumerate}[(ii)]
	\item
	Globally (on all of $D_{\pm}$) we have $\tilde{K}^{\pm}=da_{\pm}\otimes \tilde{h} + k_{\pm}$ where $\tilde{h}:=h\circ \pi$ for some $h\in \mathfrak{h}$ and each of the other ingredients are required to satisfy 
	\begin{enumerate}[(b)]
		\item
		$\epsilon_{\pm}^*\tilde{K}^{\pm}=\tilde{H}dt$ on $\{(s,t)\in Z_{\pm}\ |\ \pm s\geq 2\}$.
		\item
		$\tilde{K}^{\pm}=0$ on $D_{\pm}\backslash \image(\epsilon_{\pm})$ and $\epsilon_{\pm}^*\tilde{K}^{\pm}=0$ on $\{(s,t)\in Z_{\pm}\ |\ \pm s\leq 1\}$.
		\item
		$k_{\pm}(\xi)=0$ for all $\xi \in T\partial D_{\pm}$.
		\item
		For $C= R+1$ as in the previous subsection we have $d\pi(X_{k_{\pm}})=0$ on $(\R^2 \backslash [-C,C]^2)\times M$.
	\end{enumerate}
	\item
	$\tilde{I}^{\pm}_z=\tilde{J}'$ for all 
	$z=\epsilon_{\pm}(s,t)$ with $\pm s<1$.
	\item
	$\tilde{I}^{\pm}_z=\tilde{J}_t$ for all $z=\epsilon_{\pm}(s,t)$ with $\pm s>2$.
	\item
	$\tilde{I}^{\pm}$ satisfies the condition, that the restriction of $\pi:\tilde{M}\to \R^2$ to $(\R^2 \backslash [-C,C]^2)\times M$ is $(I^{\pm}_z,(\phi_h^{a_{\pm}(z)})_*i)$-holomorphic for all $z\in D_{\pm}$.
\end{enumerate}
We will call a perturbation datum $(\tilde{K}^{\pm},\tilde{I}^{\pm})$ satisying these specified criteria a \emph{$\PSS$-admissible perturbation datum}.
Having chosen $(\tilde{K}^{\pm},\tilde{I}^{\pm})$ we consider solutions $u_{\pm}\in C^{\infty}(D_{\pm},\R^2 \times M)$ of
\begin{equation}
\label{cob4}
\left\{
\begin{array}{l}
d_zu_{\pm}+\tilde{I}^{\pm}_z \circ d_zu_{\pm} \circ i= X_{\tilde{K}^{\pm}}+\tilde{I}^{\pm}_z \circ X_{\tilde{K}^{\pm}} \circ i\\
u_{\pm}(\partial D_{\pm})\subset V,
\end{array}
\right.
\end{equation}
where for $\xi \in T_zD_{\pm}$ the term $X_{\tilde{K}^{\pm}(\xi)}$ denotes the Hamiltonian vector field of the autonomous Hamiltonian $\tilde{K}^{\pm}(\xi)\in C^{\infty}(\tilde{M})$. The following compactness result is a small adaption of the compactness argument appearing in \cite{BiranCornea14}.
\begin{proposition}
\label{propcob2}
	Let $u_+\in C^{\infty}(D_{+},\tilde{M})$ be a solution of the "$+$"-case of (\ref{cob4}) and let $u_-\in C^{\infty}(D_{-},\tilde{M})$ be a solution of the "$-$"-case of (\ref{cob4}) satisfying the condition 
	\begin{equation}
	\label{cob8}
	u_-(1)\in [-C,C]^2\times M.
	\end{equation}
	Moreover, suppose both $u_{\pm}$ have finite energy. Then $u_{\pm}(D_{\pm})\subset [-C,C]^2\times M$.
\end{proposition} 
\begin{proof}
	The argument is the same for the two cases, so we only consider $u:=u_+:D_+\to \tilde{M}$. First note that, since $u$ has finite energy, $u\circ \epsilon_+(s,t)$ converges to a Hamiltonian chord of $\tilde{H}$ connecting $V$ to itself when $s\to \infty$. Since $\tilde{H}\in \tilde{\mathcal{H}}_{\mathfrak{h}}$ all such chords are contained in $[-C,C]^2\times M$. Now define $\tilde{h}:=h\circ \pi :\tilde{M}\to \R$ for some $h\in \mathfrak{h}$ and consider the map $v\in C^{\infty}(D_+,\tilde{M})$ defined by the equation $u(z)=\phi_{\tilde{h}}^{a_+(z)}(v(z)),\ z\in D_+$. Differentiation reveals that $v$ satisfies 
	\[
	d_zv+\tilde{I}'_z \circ d_zv \circ i= Y+\tilde{I}'_z \circ Y \circ i,
	\]
	where $X_{\tilde{K^{+}}}=d\phi_{\tilde{h}}^{a_+(z)}(Y)+da_+ \otimes X_{\tilde{h}}$ and $\tilde{I}^+_z=(\phi_{\tilde{h}}^{a_+(z)})_*\tilde{I}'_z$. Moreover, $v$ satisfies the "moving boundary condition"
	\[
	v(z)\in (\phi_{\tilde{h}}^{a_+(z)})^{-1}(V) \quad \forall \ z\in \partial D_+. 
	\]
	Note that it follows from the requirements of $\tilde{K}^{\pm}$ and $\tilde{I}^{\pm}$ that, outside the compact subset\footnote{Here the last condition imposed on our perturbation function $h$ is crucial.} 
	\begin{equation*}
		U:=[-C,C]^2\times M=\bigcup_{t\in [0,1]}(\phi_{\tilde{h}}^t)^{-1}([-C,C]^2\times M) \subset \tilde{M},
	\end{equation*}
	we have $\tilde{I}'_z=i\oplus J_z$ for some almost complex structure $J_z$ on $M$ and $d_{v(z)}\pi(Y)=0$ for all $z\in D_+$. In particular $\tilde{v}:=\pi \circ v:D_+\to \C$ restricts to a holomorphic function on $v^{-1}(\tilde{M}\backslash U)$. It then follows, using the open mapping theorem from complex analysis and the conformal properties of holomorphic maps as in the proof of Lemma 3.3.2 of \cite{BiranCornea14}, that the assumption $\tilde{v}(D_+)\cap (\R^2 \backslash [-C,C]) \neq \emptyset$ contradicts the convergence statement made in the beginning of the proof. Hence $\tilde{v}(\R\times [0,1])\subset [-C,C]^2$. Since $\phi_h^t$ preserves $[-C,C]^2$ for all $t\in [-1,1]$ the statement follows.
\end{proof}

We note that, since $-\nabla^{\tilde{\rho}}\tilde{f}$ points outwards along $\partial V|_{[-R,R]^2}$, the only relevant solutions of the "$-$"-case of (\ref{cob4}) for defining $\PSS$ are those satisfying (\ref{cob8}). Transversality issues and energy estimates for moduli spaces of such solutions are dealt with in \cite{BiranCornea14}. With these observations at hand we can define (\ref{cob2}) and (\ref{cob3}) exactly as in \cite{BiranCornea07} or \cite{Zapolsky15}, to which we refer for details. We recall that (\ref{cob2}) is defined "by counting" rigid constellations of pearly trajectories and finite energy solutions $u_+$ of (\ref{cob4}) subject to the condition that the pearly trajectory "ends" at $u_+(-1)$. (\ref{cob3}) is defined similarly. 

Following \cite{Zapolsky15} one now checks that the homology isomorphism $\PSS:QH_*(\mathcal{D}: V;\partial V)\to HF_*(\tilde{H},\tilde{J}:V)$ is independent of the chosen data and that it respects continuation isomorphisms. Moreover in \cite{Singer15} it is shown that $QH_*(V,\partial V)$ is a unital algebra, and by the standard arguments we have a canonical isomorphism $\PSS:QH_*(V,\partial V)\to HF_*(V,\mathfrak{h})$ of unital algebras.
\begin{remark}
It is important to note that the specific requirement (\ref{cob7}) imposed on the elements in $\mathfrak{h}$ in order for the PSS map $QH_*(V,\partial V)\cong HF_*(V,\mathfrak{h})$ to exist is closely connected with the definition of $QH_*(V,\partial V)$. To see the connection we suggest the curious reader take a look at the proof of Proposition 5.2 in \cite{BiranCornea13}.
\end{remark}
\begin{remark}
Note that a consequence of the above discussion is that for any two choices of perturbation functions $h^{\pm}$ satisfying (\ref{cob7}), but are in distinct classes $h^-\in \mathfrak{h}^-$, $h^+\in \mathfrak{h}^+$ there is a natural isomorphism $HF_*(V,\mathfrak{h}^-)\cong QH_*(V,\partial V) \cong HF_*(V,\mathfrak{h}^+)$ provided by $\PSS$.
\end{remark}


\subsubsection{Spectral invariants for Lagrangians with cylindrical ends}
We will apply the machinery developed in \cite{LeclercqZapolsky15} to Lagrangians with cylindrical ends. The translation to our setup is more or less immediate and we will only need a minimum of properties developed there, so we will here only mention the details needed to carry those properties over to our setup. Let $h$ be a choice of perturbation function satisfying (\ref{cob7}) and $\mathfrak{h}$ the corresponding perturbation function class. If $(\tilde{H},\tilde{J})\in \tilde{\mathcal{H}}_{\mathfrak{h}}\times \tilde{\mathcal{J}}_{\mathfrak{h}}$ is a regular Floer datum and $\alpha \in QH_*(V,\partial V)$ we define
\begin{equation}
\label{cob9}
l(\alpha, \tilde{H},\tilde{J}):=\inf\{a \in \R \ | \ \PSS(\alpha)\in \image(\iota^{a}_*)\subset HF_*(\tilde{H},\tilde{J}:V) \}.
\end{equation}
which is an element of $\R \cup \{-\infty\}$. Here $\iota^{a}_*$ denotes the map on homology induced by the natural map $\iota^{a}:CF_*^{a}(\tilde{H},\tilde{J}:V)\to CF_*(\tilde{H},\tilde{J}:V)$. It is immediate that $l(0,\tilde{H},\tilde{J})=-\infty$. For $\alpha \neq 0$ an argument from\footnote{A different setup is considered in the reference, but the argument carries over to our case \emph{mutatis mutandis}.} \cite{LeclercqZapolsky15} shows that the existence of continuation isomorphisms implies that
\begin{equation}
\label{cob11}
\int_0^1 \min_{\tilde{M}}(\tilde{H}_t^- - \tilde{H}_t^+)dt\leq  l(\alpha,\tilde{H}^-,\tilde{J}^-) -l(\alpha, \tilde{H}^+,\tilde{J}^+) \leq \int_0^1 \max_{\tilde{M}}(\tilde{H}_t^- - \tilde{H}_t^+)dt
\end{equation}
for any two regular Floer data $(\tilde{H}^-,\tilde{J}^-),(\tilde{H}^+,\tilde{J}^+)\in \tilde{\mathcal{H}}_{\mathfrak{h}}\times \tilde{\mathcal{J}}_{\mathfrak{h}}$.
In particular it follows that (\ref{cob9}) does not depend on the specific choice of compatible almost complex structure $\tilde{J}$. We therefore write $l(\alpha, \tilde{H})=l(\alpha, \tilde{H},\tilde{J})$. When we want to emphasize that $l$ is associated to the relative quantum homology $QH_*(V,\partial V)$ we write $l_{(V,\partial V)}(\alpha, \tilde{H})=l(\alpha, \tilde{H})$. Moreover, it follows from (\ref{cob11}) and genericity of non-degenerate Floer data that $l_{(V,\partial V)}$ extends by continuity to a function
 \begin{equation}
 \label{cob10}
 l_{(V,\partial V)}:QH_*(V,\partial V)\times \tilde{\mathcal{H}}_{\mathfrak{h}}\to \R \cup \{-\infty\}
 \end{equation}
 satisfying $l_{(V,\partial V)}(\alpha,\tilde{H})=-\infty$ if and only if $\alpha=0\in QH_*(V,\partial V)$.


\subsection{Proof of Theorem \ref{lemcob1}}
For the convenience of the reader the proof is split into several steps. We make use of the notation from the statement of the theorem.

\noindent
\textit{Step 1: The definition of $\Phi_V$}. 
We briefly recall the definition of the canonical restriction map $j':QH_*(V,\partial V)\to QH_{*-1}(L')$. For details we refer to Section 9 in \cite{Singer15}. Fix $R>0$ such that
\[
V|_{\R^2 \backslash [-R,R]^2}= \left((-\infty,-R]\times \{0\}\times L\right)\cup \left( [R,\infty)\times \{0\}\times L'\right).
\]
The Morse function $\tilde{f}\in C^{\infty}(V)$ in the regular quantum datum $\mathcal{D}=(\tilde{f},\tilde{\rho},\tilde{J}')$ for $QH_*(V,\partial V)$ which we consider is required to satisfy the following condition.
\begin{equation*}
\tilde{f}(t,0,p)=f^+(p) +\sigma^+(t) \quad \forall \ (t,p)\in [R,R+1]\times L'
\end{equation*} 
where $\sigma^+:[R,R+1]\to \R$ has a unique maximum at $R+\tfrac{1}{2}$ and $f^+\in C^{\infty}(L')$ is Morse. Moreover, on $[R,R+1]\times L'$ the Riemannian metric $\tilde{\rho}$ is given by $\rho \oplus \rho^+$ for some metric $\rho$ on $[R,R+1]$ and some metric $\rho^+$ on $L'$ just as well as $\tilde{J}'=i\oplus J'$ outside $[-R,R]^2\times M$ for some generic $\omega$-compatible almost complex structure $J'$ on $M$. In this setup the quotient map
\[
QC_*(\mathcal{D}: V,\partial V) \to QC_{*-1}(\mathcal{D}': L'),
\]
where $\mathcal{D}'=(f^+,\rho^+,J')$, is a chain map. The map induced on homology is exactly the map $j'$. Of course there is similarly a map $j:QH_*(V,\partial V)\to QH_{*-1}(L)$. By Theorem 2.2.2 in \cite{BiranCornea13} $V$ is a quantum h-cobordism. From Lemma 5.1.2. in the same paper it now follows that both $j$ and $j'$ are isomorphisms. That they also respect multiplication is shown in \cite{Singer15}, Theorem 1.2. By definition $\Phi_V=j'\circ j^{-1}$. The estimate in Theorem \ref{lemcob1} is therefore equivalent to the estimate
\begin{equation}
\label{cob25}
|l_L(j(\alpha),\phi)-l_{L'}(j'(\alpha),\phi)|\leq \mathcal{S}(V),
\end{equation}
for all $\alpha \in QH_*(V,\partial V)\backslash \{0\}$ and all $\phi \in \widetilde{\Ham}(M,\omega)$.
 
\noindent
\textit{Step 2: Adapting $V$.} Our strategy is based on the following trick from \cite{CorneaShelukhin15} which replaces $V$ by a new elementary Lagrangian cobordism $V':L' \rightsquigarrow L$. Fix once and for all a small $\tilde{\epsilon}>0$. Given $\psi \in \Symp(\R^2,\omega_{\R^2})$ we define $\tilde{\psi}:=\psi \times \id \in \Symp(\tilde{M},\tilde{\omega})$. We choose a $\psi$ such that every point outside $[-R,R]\times \R$ is fixed and such that $V':=\tilde{\psi}(V)$ satisfies 
\begin{equation}
\label{cob30}
\pi(V')\subset \{(x,y)\in \R^2 \ |\ 0\leq y\leq \beta(x)\},
\end{equation} 
where $\beta \in C^{\infty}_c(\R,[0,\infty))$ satisfies $\supp(\beta)\subset (-R,R)$ and  
\begin{equation}
\label{cob31}
\int_{-\infty}^{\infty}\beta(t)\ dt \leq \mathcal{S}(V')+\tilde{\epsilon}.
\end{equation}
The existence of such $\psi$ and $\beta$ is quite obvious. Moreover, the construction implies that $V':L' \rightsquigarrow L$ is an elementary Lagrangian cobordism satisfying $\mathcal{S}(V)=\mathcal{S}(V')$ and $\Phi_V=\Phi_{V'}$.

\noindent
\textit{Step 3: Constructing suitable extensions of $H$.} We first fix a perturbation function $h$ satisfying the following criteria (here we use the notation from the first conditions (i)-(iv) in Section \ref{HFcylends})
\begin{itemize}
	\item
	$\supp(h)\subset V_0^- \cup V_0^+$ and $\partial_yh(x,y)=0$ for all $(x,y)$ satisfying $|y|<\tfrac{\epsilon}{2}$.
	\item
	$h$ must satisfy (\ref{cob7}) as well as   
	\begin{align*}
	\partial_xh(x,0)&\geq 0 \quad \forall \ x\leq - R \\
	\partial_xh(x,0)&\leq 0 \quad \forall \ x\geq R.
	\end{align*}  
	Moreover, we require these inequalities be strict for $x<-(R+\tfrac{1}{3})$ and $x>R+\tfrac{1}{3}$ respectively.
	\item
	Lastly, we require $|h(-(R+\tfrac{1}{2}),0)|,|h(R+\tfrac{1}{2},0)|\leq \tilde{\epsilon}$. 
\end{itemize}
Denote by $\mathfrak{h}$ the class corresponding to $h$. Fix now $0<\delta <\! \! < \tilde{\epsilon}$ and choose a cut-off $b\in C^{\infty}_c(\R,[0,\delta])$ such that
\begin{equation*}
b=
\left\{
\begin{array}{ll}
\delta, & \text{on}\ [-R-\tfrac{1}{2},R+\tfrac{1}{2}] \\
0, & \text{on}\ \R \backslash (-C+\tfrac{1}{5},C-\tfrac{1}{5}).
\end{array}
\right.
\end{equation*}  
Define the constant $K:=\int_{-R-1}^{R+1}\left(b(t)+\beta(t) \right)\ dt$ which for small enough choice of $\delta$ satisfies $K\leq \mathcal{S}(V')+2\tilde{\epsilon}$. We now define 3 auxiliary functions as follows. We denote by $\rho \in C^{\infty}_c(\R,[0,1])$ a cut-off satisfying  
\begin{equation*}
\rho=
\left\{
\begin{array}{ll}
1, & \text{on}\ [-R,R] \\
0, & \text{on}\ \R \backslash (-(R+\tfrac{1}{2}),R+\tfrac{1}{2}).
\end{array}
\right.
\end{equation*}
By $\eta_-,\eta_+ \in C^{\infty}(\R,[0,1])$ we denote \emph{monotone} functions satisfying
\begin{equation*}
\eta_-=
\left\{
\begin{array}{ll}
0, & \text{on}\ (-\infty,-R-\tfrac{4}{5}] \\
-1, & \text{on}\ [-R-\tfrac{2}{3},\infty),
\end{array}
\right.
\quad
\&
\quad 
\eta_+=
\left\{
\begin{array}{ll}
1, & \text{on}\ (-\infty,R+\tfrac{2}{3}] \\
0, & \text{on}\ [R+\tfrac{4}{5},\infty),
\end{array}
\right.
\end{equation*}
as well as $\max_{\R}|\eta_-'|,\max_{\R}|\eta_+'|\leq 10$. Using this auxiliary data we can finally define two specific perturbation functions $h_-,h_+\in \mathfrak{h}$ by
\begin{align*}
	h_+(x,y)&=\left(\int_{-R-1}^{x}(b(t)+\beta(t))\ dt\right)\eta_+(x)\rho(y) +h(x,y)\\
	h_-(x,y)&=\left(\int_{-R-1}^{x}(b(t)+\beta(t))\ dt-K \right)\eta_-(x)\rho(y) +h(x,y),
\end{align*}
where $(x,y)\in \R^2$. We have constructed $h_\pm$ such that they satisfy the following properties. Assuming our data is chosen carefully we achieve
\begin{equation*}
\phi_{h_-}^1(\pi(V'))\cap \pi(V')=\{(-R-\tfrac{1}{2},0)\} \quad \& \quad \phi_{h_+}^1(\pi(V'))\cap \pi(V')=\{(R+\tfrac{1}{2},0)\}.
\end{equation*}
Moreover, $x\mapsto h_-(x,0)$ has a local non-degenerate maximum at $x=-(R+\tfrac{1}{2})$ and $x\mapsto h_+(x,0)$ has a local non-degenerate maximum at $x=R+\tfrac{1}{2}$. Given any Hamiltonian $H\in C_c^{\infty}([0,1]\times M)$ we define associated Hamiltonians $\tilde{H}^-,\tilde{H}^+ \in \tilde{\mathcal{H}}_{\mathfrak{h}}$ by
\begin{align*}
	\tilde{H}^-_t(z,p)=H_t(p)+h_-(z) \quad \& \quad \tilde{H}^+_t(z,p)=H_t(p)+h_+(z),
\end{align*}
for $(t,z,p)\in [0,1]\times \R^2 \times M$. For future use we denote $c_{\pm}:=h_{\pm}(\pm (R+\tfrac{1}{2}))$. Note that we can estimate
\begin{equation}
\label{cob40}
	|c_--c_+|\leq 5\tilde{\epsilon}.
\end{equation}	

\begin{figure}[htbp]
\centering
\includegraphics[scale=0.9, clip=true, trim=3.4cm 23cm 4.7cm 2.8cm]{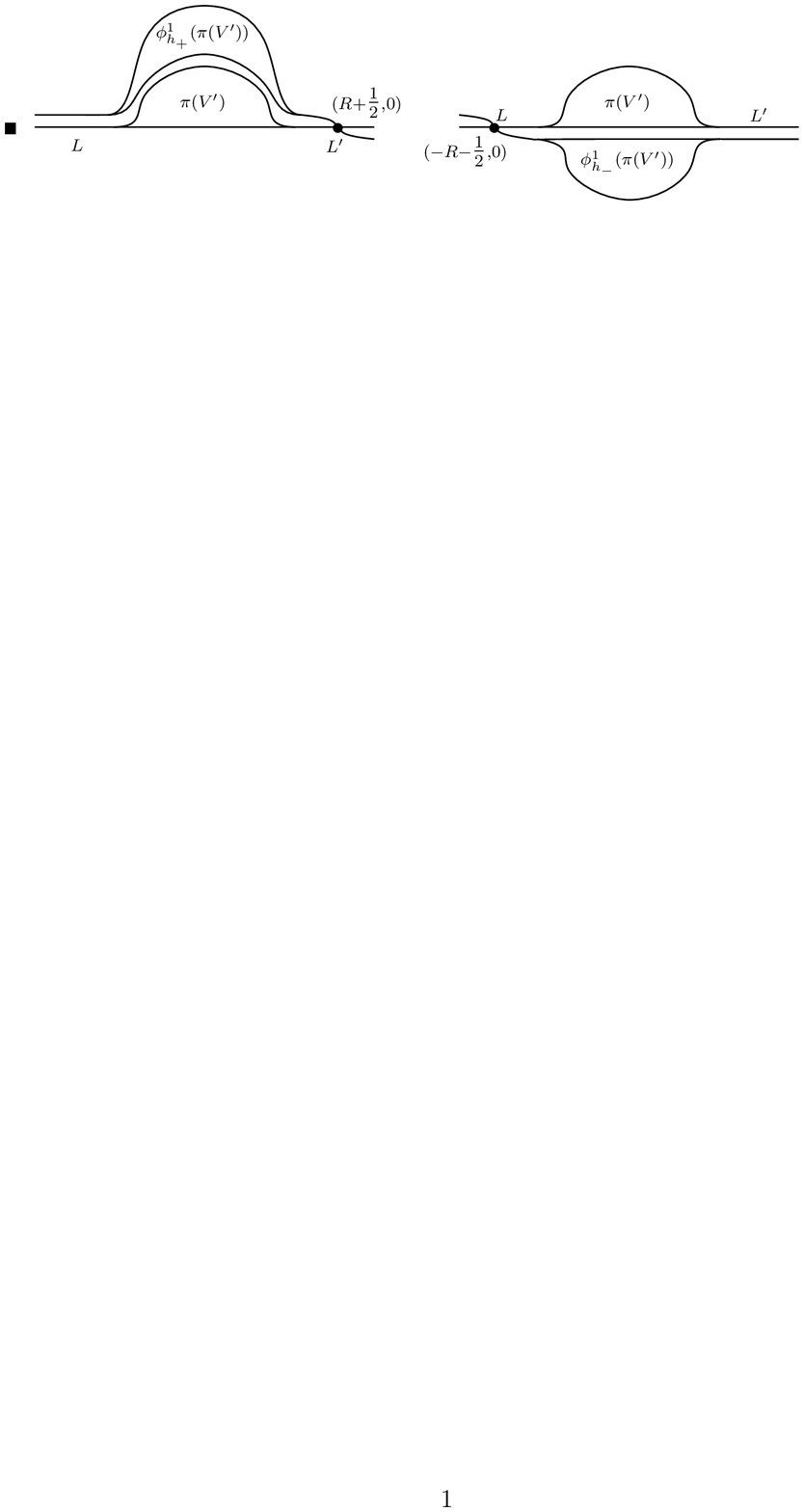}
\caption{The main point about our construction of $h_{\pm}$ is that they are elements of $\mathfrak{h}$ and their Hamiltonian flows have the following "effect": In the left picture the image of $\pi(V')$ under $\phi_{h_+}^1$ is indicated. In the right picture the image under $\phi_{h_-}^1$.}
\label{fig1}
\end{figure}

\noindent
\textit{Step 4: Relating spectral invariants of the ends to those of $V'$.}
Let $H\in C_c^{\infty}([0,1]\times M)$ be a normalized Hamiltonian which is non-degenerate both for $L$ and $L'$. Define $\phi:=\phi_H\in \widetilde{\Ham}(M,\omega)$. It suffices to prove the theorem for such $\phi$. In the notation of \cite{LeclercqZapolsky15} we have $l_L(\alpha,H)=l_L(\alpha,\phi)$ for $\alpha \in QH_*(L)$.
\begin{lemma}
	For all $\alpha \in QH_*(V',\partial V')$ it holds that
	\[
	l_{L'}(j'(\alpha),H)=l_{(V',\partial V')}(\alpha, \tilde{H}^+)-c_+ \quad \& \quad l_L(j(\alpha),H)=l_{(V',\partial V')}(\alpha, \tilde{H}^-)-c_-.
	\]
\end{lemma}
\begin{proof}
	The proofs of the two equalities are similar, so we only prove the first one. Fix once and for all a generic path of $\omega$-compatible almost complex structures $J=\{J_t\}_{t\in [0,1]}$ in $M$ so that the Floer chain complex $(CF_*(H,J:L'),d)$ is well defined. We also want to consider the Floer chain complex $CF_*(\tilde{H}^+,\tilde{J}:V')$ for a specific choice of regular $\tilde{J}$. We specify this choice below. Note that, by the construction of $h_{+}$ there is a 1-1 correspondence between chords of $X_{\tilde{H}^+}$ connecting $V'$ to itself and chords of $X_H$ connecting $L'$ to itself (see Figure \ref{fig1}). 
	We define the subspace $Y_*\subset CF_*(\tilde{H}^+,\tilde{J}:V')$ as the $\Z_2$-vector space generated by those critical points\footnote{Recall that a critical point $[\gamma,\widehat{\gamma}]$ of $\mathcal{A}_{\tilde{H}^+:V'}$ consists of a chord $\gamma:([0,1],\{0,1\})\to (\tilde{M},V')$ satisfying $\dot{\gamma}=X_{\tilde{H}^+}(\gamma)$ and an equivalence class of cappings $\widehat{\gamma}$ of $\gamma$, where we say that two cappings of $\gamma$ are equivalent if they have the same symplectic area.} $[\gamma,\widehat{\gamma}]$ of $\mathcal{A}_{\tilde{H}^+:V'}$ for which $\widehat{\gamma}$ is \emph{not} equivalent to any capping of $\gamma$ whose image is completely contained in the fiber $\{(R+\frac{1}{2},0)\}\times M$. Since any capping of a $X_H$-chord sitting inside $M$ can be viewed as a capping of the corresponding $X_{\tilde{H}^+}$-chord sitting inside $\{(R+\tfrac{1}{2},0)\}\times M$ there is a well-defined inclusion map\footnote{The fact that $\iota$ increases the degree by 1 comes from the fact that we follow the normalization convention of the Conley-Zehnder index from \cite{Zapolsky15}. We point out that this convention corresponds to assigning Maslov index 1 to the loop $\R/\Z \ni t\mapsto e^{-t\pi i}\R$ in the Lagrangian Grassmannian of $(\C,dx\wedge dy)\approx (\R^2, \omega_{\R^2})$.}
	\begin{equation}
	\label{cob27}
	\iota :CF_{*-1}(H,J:L')\hookrightarrow CF_*(\tilde{H}^+,\tilde{J}:V')
	\end{equation}
	of $\Z_2$-vector spaces. It is clear from the definition of $Y_*$ that the image of $\iota$ is a direct complement to $Y_*$. I.e. we obtain a splitting of $\Z_2$-vector spaces:
	\begin{equation}
	\label{cob26}
	CF_*(\tilde{H}^+,\tilde{J}:V')=CF_{*-1}(H,J:L')\oplus Y_*.
	\end{equation}
Now choose any $\tilde{J}=\{\tilde{J}_t\}_{t\in [0,1]}\in \tilde{\mathcal{J}}_{\mathfrak{h}}$ satisfying the condition that $\tilde{J}_t=(\phi_{h_+}^t)_*i \oplus J_t$ outside $[-R,R]^2\times M$. We claim that $\tilde{J}$ is regular for $\tilde{H}^+$. To see this, note that all finite energy Floer trajectories corresponding to the data $(\tilde{H}^+,\tilde{J})$ are completely contained in the fiber $\{(R+\tfrac{1}{2},0)\}\times M$. This is easy to see using the same trick as in the proof of Proposition \ref{propcob2} but is in fact also a simple case of the \emph{bottleneck} construction in the proof of Lemma 3.3.2 in \cite{BiranCornea14}. In particular the linearized operator associated to Floer's equation splits along any finite energy Floer trajectory. It therefore follows from the automatic transversality result in Corollary 4.3.2 from \cite{BiranCornea14}, which is based on the theory developed in \cite{Seidel12}, that $(\tilde{H}^+,\tilde{J})$ is a regular Floer datum. Here it was crucial that $x=R+\tfrac{1}{2}$ is a local maximum for $\R \ni x\mapsto h_+(x,0)$. Considering the definition of $Y_*$ it is not hard to see that this implies that (\ref{cob26}) is in fact a splitting of chain complexes. As a consequence the quotient map $q:(CH_*(\tilde{H}^+,\tilde{J}:V'),d)\to (CH_{*-1}(H,J:L'),d)$ which collapses the $Y_*$ component is a chain map satisfying 
\begin{equation}
\label{cob41}
q(CF_*^{a+c_+}(\tilde{H}^+,\tilde{J}:V'))\subset CF_{*-1}^a(H,J:L')\quad \forall \ a\in \R.	
\end{equation} 
Similarly one sees that
\begin{equation}
\label{cob42}
\iota(CF_{*-1}^a(H,J:L'))\subset CF_*^{a+c_+}(\tilde{H}^+,\tilde{J}:V') \quad \forall \ a\in \R.
\end{equation}
We now want to choose regular perturbation data for $\PSS_{V'}$ and $\PSS_{L'}$ such that the two diagrams
\begin{align*}
	\xymatrix{
		CF_*(\tilde{H}^+,\tilde{J}:V')  & CF_{*-1}(H,J:L') \ar[l]_{\iota} \\
		QC_*(\mathcal{D}: V',\partial V') \ar[r]^{j'} \ar[u]^{\PSS_{V'}} & QC_{*-1}(\mathcal{D}': L') \ar[u]_{\PSS_{L'}} 
		} \\
	\xymatrix{
		CF_*(\tilde{H}^+,\tilde{J}:V')  \ar[r]^{q} & CF_{*-1}(H,J:L') \\
		QC_*(\mathcal{D}: V',\partial V') \ar[r]^{j'} \ar[u]^{\PSS_{V'}} & QC_{*-1}(\mathcal{D}': L') \ar[u]_{\PSS_{L'}}
		}
\end{align*}
commute. Denote by $(K^+,I^+)$ a regular perturbation datum for $\PSS_{L'}:QC_*(\mathcal{D}': L')\to CF_*(H,J:L')$. By definition of $j'$ we see\footnote{Lemma 1.3 in \cite{Singer15} implies that pearly trajectories which are not completely contained in the fiber $\{(R+\tfrac{1}{2},0)\}\times M$ cannot "end in it".} that in order to accomplish commutativity of the above diagrams it suffices to extend $(K^+,I^+)$ to a regular perturbation datum $(\tilde{K}^+,\tilde{I}^+)$ such that the image of every finite energy solution $u_+\in C^{\infty}(D_+,\tilde{M})$ of (\ref{cob4}) is contained in the fiber $\{(R+\tfrac{1}{2},0)\}\times M$ and can be identified with a solution of the corresponding equation for $\PSS_{L'}:QC_*(\mathcal{D}': L')\to CF_*(H,J:L')$. Once this has been carried out the claim easily follows from (\ref{cob41}) and (\ref{cob42}).
	
We now argue how to extend $(K^+,I^+)$. First choose a $\PSS$-admissible family of almost complex structures $\{\tilde{I}^+_z\}_{z\in D_+}$ on $\tilde{M}$ satisfying the condition that 
\begin{align*}
\tilde{I}^+_z|_{(\R^2 \backslash [-R,R]^2)\times M}&=(\phi_{h_+}^{a_+(z)})_*i \oplus I^+_z\quad \forall \ z\in D_+.
\end{align*}
Now define
\begin{equation*}
\tilde{K}^+:= da_+ \otimes \tilde{h}_+ + k_+,
\end{equation*}
where $k_+ \in \Omega^1(D_+,C^{\infty}(\tilde{M}))$ is defined by $k_+(\xi)=K^+(\xi)$ for $\xi \in TD_+$. Needless to say, we here view $K^+ \in \Omega^1(D_+,C^{\infty}(\tilde{M}))$ in the obvious way. Moreover $\tilde{h}_+:=h_+\circ \pi$. We note that, using this data, it follows from the same argument as above that the image of any finite energy solutions $u_+$ of (\ref{cob4}) is contained in the fiber $\{(R+\tfrac{1}{2},0)\}\times M$. 
Hence, the linearization of the operator associated to (\ref{cob4}) along any finite energy solution is split. One last time we use Lemma 4.3.1 in \cite{BiranCornea14}, for the case $k=0$ in their terminology, to argue that the perturbation datum $(\tilde{K}^+,\tilde{I}^+)$ is regular.\footnote{Technically speaking the case $k=0$ is not contained in the statement of Lemma 4.3.1 in \cite{BiranCornea14}. However, applying the methods from \cite{Seidel12} exactly as in the proof of this lemma one sees that the conclusion of the lemma also holds for the case $k=0$.} Again it is crucial that $x=R+\tfrac{1}{2}$ is a local maximum for $x\mapsto h_+(x,0)$. This finishes the proof of the claim. 
\end{proof}

\noindent
\textit{Step 5: The estimate.}
Applying (\ref{cob11}) and making use of the lemma together with the estimates obtained in Step 3 one easily computes that, for all $\alpha \in QH_*(V',\partial V')\backslash \{0\}$, we have
\begin{align}
\label{cob28}
|l_{L}(j(\alpha),\phi)-l_{L'}(j'(\alpha),\phi)|&\leq |l_{(V',\partial V')}(\alpha,\tilde{H}^-)-l_{(V',\partial V')}(\alpha,\tilde{H}^+)| +5\tilde{\epsilon}  \\
&\leq \int_0^1 \max_{\tilde{M}}|\tilde{H}^-_t-\tilde{H}^+_t|\ dt +5\tilde{\epsilon} \nonumber \\
&=\max_{[-C,C]^2}|h_--h_+| +5\tilde{\epsilon} \nonumber \\
&\leq \mathcal{S}(V')+10\tilde{\epsilon}=\mathcal{S}(V)+10\tilde{\epsilon}. \nonumber
\end{align}
Since $\tilde{\epsilon}>0$ was arbitrary the proof is done. 

\bibliographystyle{plain}
\bibliography{BIBpub}
\end{document}